\documentclass[12pt]{amsart}
\usepackage{amscd,amsmath,amsthm,amssymb}
\usepackage[left]{lineno}
\usepackage{color}
\usepackage{stmaryrd}
\usepackage[utf8]{inputenc}
\usepackage{cleveref}
\usepackage{epstopdf}
\usepackage{graphicx}
\usepackage{xcolor}

\usepackage{tikz}

\definecolor{verylight}{gray}{0.97}
\definecolor{light}{gray}{0.9}
\definecolor{medium}{gray}{0.85}
\definecolor{dark}{gray}{0.6}

 \def\NZQ{\mathbb}               

 \def\KK{{\NZQ K}}
 
 \def\frk{\mathfrak}               

 \def\mm{{\frk m}}

\def\KK{{\NZQ K}}

 \def\G{{\mathcal G}}

\def\Ht{\operatorname{ht}} 
 \def\0b{{\mathbf 0}}

\def\reg{{\mathbf reg}}
\def\height{\operatorname{ht}}
\def\depth{\operatorname{depth}}

\def\Char{\operatorname{char}}

 \def\opn#1#2{\def#1{\operatorname{#2}}} 
 \opn\chara{char} \opn\length{\ell} \opn\pd{pd} \opn\rk{rk}
 \opn\projdim{proj\,dim} \opn\injdim{inj\,dim} \opn\rank{rank}
 \opn\depth{depth} \opn\grade{grade} \opn\height{height}
 \opn\embdim{emb\,dim} \opn\codim{codim}
 
 \opn\Tr{Tr} \opn\bigrank{big\,rank}
 \opn\superheight{superheight}\opn\lcm{lcm}
 \opn\trdeg{tr\,deg}
 \opn\reg{reg} \opn\lreg{lreg} \opn\ini{in} \opn\lpd{lpd}
 \opn\size{size} \opn\sdepth{sdepth}
 \opn\link{link}\opn\fdepth{fdepth}\opn\lex{lex}
 \opn\tr{tr}
 \opn\type{type}
 \opn\gap{gap}
 \opn\arithdeg{arith-deg}
 \opn\HS{HS}
 \opn\GL{GL}
 \opn\div{div} \opn\Div{Div} \opn\cl{cl} \opn\Cl{Cl}
 \opn\Spec{Spec} \opn\Supp{Supp} \opn\supp{supp} \opn\Sing{Sing}
 \opn\Ass{Ass} \opn\Min{Min}\opn\Mon{Mon}
 \opn\Ann{Ann} \opn\Rad{Rad} \opn\Soc{Soc}\opn\Deg{Deg}
 \opn\Im{Im} \opn\Ker{Ker} \opn\Coker{Coker} \opn\Am{Am}
 \opn\Hom{Hom} \opn\Tor{Tor} \opn\Ext{Ext} \opn\End{End}
 \opn\Aut{Aut} \opn\id{id}
 
 \opn\nat{nat}
 \opn\pff{pf}
 \opn\Pf{Pf} \opn\GL{GL} \opn\SL{SL} \opn\mod{mod} \opn\ord{ord}
 \opn\Gin{Gin} \opn\Hilb{Hilb}\opn\sort{sort}
 \opn\PF{PF}\opn\Ap{Ap}
 \opn\mult{mult}
 \opn\bight{bight}
 \opn\aff{aff}
 \opn\relint{relint} \opn\st{st}
 \opn\lk{lk} \opn\cn{cn} \opn\core{core} \opn\vol{vol}  \opn\inp{inp} \opn\nilpot{nilpot}
 \opn\link{link} \opn\star{star}\opn\lex{lex}\opn\set{set}
 \opn\width{wd}
 \opn\Fr{F}
 \opn\QF{QF}
 \opn\G{G}
 \opn\type{type}\opn\res{res}
 \opn\conv{conv}
 \opn\Ind{Ind}
 \opn\gr{gr}
 
 \def\pot#1#2{#1[\kern-0.28ex[#2]\kern-0.28ex]}

 \opn\dirlim{\underrightarrow{\lim}}
 \opn\inivlim{\underleftarrow{\lim}}
 \def\Implies{\ifmmode\Longrightarrow \else
         \unskip${}\Longrightarrow{}$\ignorespaces\fi}
 \def\implies{\ifmmode\Rightarrow \else
         \unskip${}\Rightarrow{}$\ignorespaces\fi}
 \def\iff{\ifmmode\Longleftrightarrow \else
         \unskip${}\Longleftrightarrow{}$\ignorespaces\fi}

 \let\:=\colon
 \newtheorem{Theorem}{Theorem}[section]
 \newtheorem{Lemma}[Theorem]{Lemma}
 \newtheorem{Corollary}[Theorem]{Corollary}
 
 \newtheorem{Remark}[Theorem]{Remark}
 
 \newtheorem{Example}[Theorem]{Example}
 
 \newtheorem{Definition}[Theorem]{Definition}

 \let\epsilon\varepsilon
 \let\kappa=\varkappa
 \def\qed{\ifhmode\textqed\fi
       \ifmmode\ifinner\quad\qedsymbol\else\dispqed\fi\fi}
 \def\textqed{\unskip\nobreak\penalty50
        \hskip2em\hbox{}\nobreak\hfil\qedsymbol
        \parfillskip=0pt \finalhyphendemerits=0}
 \def\dispqed{\rlap{\qquad\qedsymbol}}
 
 \opn\dis{dis}
 \def\pnt{{\raise0.5mm\hbox{\large\bf.}}}
 
 \opn\Lex{Lex}

\begin{document}

\title{The edge ideals of the join of some vertex weighted oriented graphs}

\author{Yijun Cui$^{1}$,    Guangjun Zhu$^{\ast\,2}$ and Xiaoqi Wei$^{3}$}

\address{$^{1,2}$ School of Mathematical Sciences, Soochow University, Suzhou, Jiangsu, 215006, P. R. China}

\address{$^{3}$ School of Mathematics and Physics, Jiangsu University of Technology, Changzhou, Jiangsu 213001, P. R. China}

\email{237546805@qq.com(Yijun Cui), zhuguangjun@suda.edu.cn(Corresponding author:Guangjun Zhu), weixq@jsut.edu.cn(Xiaoqi Wei).}

\thanks{$^{\ast}$ Corresponding author}


\begin{abstract}
   In this paper,  we describe primary decomposition of the edge ideal of the join of some graphs  in terms of that information of the edge ideal of every weighted oriented graph.
Meanwhile, we also study depth and regularity of symbolic powers and ordinary powers
of  such an edge ideal.  We  explicitly compute  depth and regularity of ordinary
powers of the edge ideal of the join of two  graphs consisting of isolated vertices,  and also  provide   upper bounds of regularity of
 symbolic powers of such an  edge ideal.   For the  edge ideal of the join of two graphs
 with at least an oriented edge for per graph, we give the exact formulas for their depth and regularity, and also provide the upper bounds of
regularity of ordinary powers of such an  edge ideal. Some examples show that these upper bounds can be
obtained, but may be strict.
\end{abstract}

\maketitle

\section{Introduction}
Let $G=(V(G),E(G))$ be a finite simple (no loops, no multiple edges) undirected graph.
A {\em  weighted oriented graph}  $D$ whose underlying graph is $G$, is a triplet $(V(D),E(D),w)$
where $V(D)=V(G)$ is the vertex set, $E(D)$ is a directed edge set and $w$ is a function $w: V(D)\rightarrow \mathbb{N}^{+}$, where $\mathbb{N}^{+}=\{1,2,\ldots\}$.
 Specifically, $E(D)$  consists of ordered pairs $(x_i,x_j)\in V(D)\times V(D)$ where the pair $(x_i,x_j)$
represents a directed edge from $x_i$ to $x_j$.
Some times for short we denote  $V(D)$ and  $E(D)$
by $V$ and $E$ respectively. For any $x_i\in V$, its weight  is  denoted by  $w(x_i)$ or $w_i$.
For any $x_i\in V(D)$,  the sets $N^+_D(x_i)=\{x_j\,|\, (x_i,x_j)\in E(D)\}$ and $N^-_D(x_i)=\{x_j\,|\, (x_j,x_i)\in E(D)\}$
 are called the {\em out-neighbourhood} and  {\em in-neighbourhood} of $x_i$, respectively.  Furthermore,
the set $N_D(x_i)=N^+_D(x_i)\cup N^-_D(x_i)$   is  called the {\em neighbourhood} of $x_i$.
 If $N_D(x_i)=N^+_D(x_i)$, then  $x_i$ is called a source of $D$; If $N_D(x_i)=N^-_D(x_i)$, then  $x_i$ is called a sink  of $D$.
Define $deg_D(x_i)=|N_D(x_i)|$  and $N_D[x_i]=N_D(x_i)\cup \{x_i\}$.

Let $D=(V(D),E(D),w)$  denote a weighted oriented graph with vertices $V(D)=\{x_{1},\dots,x_{n}\}$.
By identifying the vertices with the variables in the polynomial ring  $S=\KK[x_1,\ldots,x_n]$  over a field $\KK$.
we can associate to each weighted oriented graph  $D$  a monomial ideal
\[
I(D)=(x_ix_j^{w(x_j)}\mid  (x_i, x_j)\in E(D)).
\]
 This ideal is called the {\em  edge ideal} of $D$. The
generators of $I(D)$ are independent of the weight assigned to a source vertex. Therefore, to simplify our
formulas, throughout this paper, we shall assume that source vertices always have weight one.

Let $S=\KK[x_1,\ldots,x_n]$ be a polynomial ring in $n$ variables over a field $\KK$ and  $I$ be a nonzero
homogeneous ideal of $S$. Then for  $k\ge 1$,  the $k^{th}$ symbolic power of $I$  is defined as
$I^{(k)}= \bigcap \limits_{P \in \Ass(I)}(I^{k}S_{P})\cap S$,  where $\Ass(I)$ is the set
of associated prime ideals of  $I$.  Geometrically, the symbolic powers are important since
they capture all the polynomials that vanish with a given multiplicity (see \cite{DDGHN}). It is clear that $I^k\subseteq I^{(k)}$ for all $k\ge 1$ but the reverse containment may fail.
We know that the regularity and depth are two central invariants associated to $I$.
It is well known that $\reg(I^t)$ is asymptotically a linear function for $t\gg  0$, i.e., there exist constants $a$,  $b$ and a positive integer $t_0$ such that for all $t\geq t_0$, $\mbox{reg}\,(I^t)=at+b$ (see \cite{CHT,K}).
It was proved that if $I=I(G)$ is an edge ideal of a  simple graph, then $a\le 2$.
 Kumar et al. \cite{KKS} studied the regularity of  symbolic powers of edge ideals of the join of simple graphs, they show that:
\begin{Theorem}
 Let $r\ge 2$  be an integer and  $G_1*\cdots*G_r$ be the join of $r$ simple graphs $G_1,\ldots,G_r$ with   pairwise  disjoint vertex sets $V_1,\ldots, V_r$, respectively. Then
  \[
 \reg\,(S/I(G)^{(t)})=\max\{\reg\,(S/I(G_j)^{(i)})-i+t\mid 1\le i\le t,1\le j\le r\},  \text{ for all $t\ge 1$}.
\]
  \end{Theorem}
In \cite{S}, Selvaraja defined following classes of graphs:
\[
\mathcal{A}=\{G\mid \reg(S/(I(G)^{t+1}: u))\le \reg(S/I(G)),u \in \mathcal{G}(I(G)^t),t\ge 1\},
\]
where $\mathcal{G}(I(G)^t)$  denotes the minimal monomial generating set of $I(G)^t$. He show that:
\begin{Theorem}
 Let  $G_1,G_2\in \mathcal{A}$ be two graphs with disjoint vertex sets.Then for all $t\ge 1$, then
  \[
 \reg\,(S/I(G_1*G_2)^{t})\le 2t+\reg\,(S/I(G_1*G_2))-2.
\]
  \end{Theorem}
As far as we know, little is known about how to calculate  regularities of  symbolic powers and  ordinary  powers of edge ideals of some weighted oriented graphs (see \cite{ WZX,XZWZ,ZXWT,ZXWZ}).

In \cite{B}, Brodmann  showed that  $\mbox{depth}\,(S/I^t)$ is a constant for $t\gg 0$, and this constant is bounded
above by $n-\ell(I)$, where $\ell(I)$ is the analytic
spread of $I$.  It is shown  in \cite[Theorem 1.2]{HH2} that $\mbox{depth}\,(S/I^t)$ is a nonincreasing
function of $t$ when all powers of $I$ have a linear resolution
and conditions are given in that paper under which all powers of $I$ will
have linear quotients.
In this regard, there has been an interest in determining the smallest value $t_0$ such that $\mbox{depth}\,(S/I^t)$ is a constant for all $t\geq t_ 0$.
 (see  \cite{FM,HH2,M, WZX,XZWZ,ZXWT,ZXWZ}).

In this article, we focus on algebraic properties corresponding to
the irreducible decomposition (see Theorem \ref{rdecompose}),  depth and  regularity of symbolic powers and ordinary powers
of  edge ideals of the join of some weighted oriented graphs.

Let  $r\ge 2$ be an integer and  $D_1,\ldots,D_r$ be weighted oriented graphs over   pairwise  disjoint vertex sets $V_1,\ldots, V_r$, respectively. The join of  $D_1,\ldots,D_r$, denoted by  $D_1*\cdots*D_r$,  is a weighted oriented graph over the vertex set   $V_1\sqcup \cdots \sqcup V_r$,  whose  edge set is
\[
E_1\cup\cdots \cup E_r\cup \{(x, y): x\in V_i, y\in  V_{j}\text {\ with\ }i<j\}.
\]
In particular, if $r=2$, then the  join $D_1*D_2$ of   $D_1$ and $D_2$  is a weighted oriented graph over the vertex set  $V_1\sqcup V_2$ whose  edge set is $E_1\cup E_2\cup \{(x,y) : x\in V_1, y\in  V_2\}$.

 Our main results are as follows:
\begin{Theorem}
	Let $D:=D_1*D_2$ be the join  of two  weighted oriented graphs  $D_1$ and $D_2$, where $D_i$ consists of isolated vertices with   $V_i=\{x_{ij}\,|\,  j\in [n_i]\}$ for   $i=1,2$.
	Then, for all $t$, we have
\begin{itemize}
	\item[(1)] 		$\depth\,(S/I(D)^t)=\depth\,(S/I(D)^{(t)})=1$;
	\item[(2)]  $\reg\,(S/I(D)^t)=\reg\,(S/I(D)^{(t)})= \sum\limits_{x\in V(D)}{w(x)}-|V(D)|+1+(t-1)(w+1)$
\end{itemize}
where  $w=\max\{w(x) \,|\, x\in V(D)\}$.
\end{Theorem}

\begin{Theorem}
Let $r\geq2$ be an integer and let   $D:=D_1*\cdots*D_r$ be the join  of   weighted oriented graphs  $D_1,\ldots,D_r$, where $D_i$ consists of isolated vertices with  $V_i=\{x_{ij}\,|\,  j\in [n_i]\}$ for any  $i\in [r]$.
Then, for any $t\ge 1$, we have
\begin{itemize}
	\item[(1)] 		$\depth\,(S/I(D)^{(t)})=1$;
	\item[(2)]  $\reg\,(S/I(D)^{(t)})\leq \sum\limits_{x\in V}{w(x)}-|V|+1+(t-1)(w+1)$.
\end{itemize}
The	equality holds  if $w=\max\{w(x) \,|\, x\in V_2\}$, where  $w=\max\{w(x) \,|\, x\in V\}$ and $V=V_1\sqcup\cdots\sqcup V_r$.
\end{Theorem}

\begin{Theorem}\label{r-partitedirected}
	Let $r\geq3$ be an integer, and let $D:=D_1*\cdots*D_r$ be the join of  weighted oriented graphs  $D_1,\ldots,D_r$,  where the vertex set of $D_i$ is  $V_i=\{x_{ij}\,|\, j\in [n_i]\}$ and $n_i=|V_i|$ for  $i\in [r]$. If each $D_i$ contains at least an  oriented edge. Then
	\begin{itemize}
		\item[(1)] $\depth\,(S/I(D))=1$;
		\item[(2)]  $\reg\,(S/I(D))\!=\!\max\{\reg\,(S_i/I(D_i))+\sum\limits_{x\in T_i}w(x)-|T_i|:i\in [r]\}$, where   $T_i=V(D)\setminus \bigcup\limits_{j=1}^i V_j$ for any  $i\in [r]$.
\end{itemize}	
\end{Theorem}	

\begin{Theorem}
Let $D:=D_1*D_2$ be the join of two weighted oriented graphs  $D_1$ and $D_2$, where $D_1$ contains at least an  oriented edge and $D_2$ is a weighted oriented complete graph.
If $\reg\,(S_1/I(D_1)^t)\leq \reg\,(S_1/I(D_1))+(t-1)(w'+1)$, where $S_1=\KK[x_{1i}: i\in V(D_1)]$ and $w'=\max\{w(x) \,|\, x\in V(D_1)\}$.
	Then, for all $t\ge 1$, we have
	\[
	\reg\,(S/I(D)^t)\leq \reg\,(S/I(D))+(t-1)(w+1)
	\]
	where $w=\max\{w(x) \,|\, x\in V(D_1)\cup V(D_2)\}$. The  equality holds  when $w'=w$ and $\reg\,(S_1/(I(D_1)^t))= \reg\,(S_1/I(D_1))+(t-1)(w'+1)$.
\end{Theorem}
Our results partially generalize the corresponding conclusion of  symbolic powers and ordinary powers
of  edge ideals of the join of some  simple  graphs, since if $w(x)=1$ for all $x\in V(D)$, then
$I(D)=I(G)$.

Our paper is organized as follows. In the preliminary section, we collect the needed notations and basic facts from the literature.
In section $2$, we describe primary decomposition of the edge ideal of the join of some graphs  in terms of that information of the edge ideal of every weighted oriented graph.
In section $3$, we study  depth and regularity of symbolic powers and ordinary
powers of the edge ideal of the join of weighted oriented graphs, where each graph
consists of isolated vertices. We provide some exact formulas for
 depth and regularity of ordinary powers  and also provide upper bounds of regularity
of symbolic powers of such an edge ideal. In section $4$, we study  depth and regularity of symbolic powers and ordinary
powers of the edge ideal of the join of weighted oriented graphs, where each graph contains at least an  oriented edge, we give upper bounds of regularity
of ordinary powers of such an edge ideal.

\medskip
\section{Preliminaries}
In this section, we gather together the needed notations and basic facts, which will be used throughout this paper.
Two important invariants  we focus on  are depth and  regularity,
we  define them by means of local cohomology modules.

Let $S$ be a positively graded algebra and  $\mm$  its maximal homogeneous ideal. Let $M$ be a finitely generated graded $S$-module. Let $H_{\mm}^{i}(M)$, for $i\geq 0$,
denote the $i$-th local cohomology module of $M$ with respect to $\mm$. We define
	\begin{align*}
\depth\,(M)&=\text{min}\,\{i\,|\,  H_{\mm}^{i}(M)\neq 0\},\\
\reg\,(M)&= \text{ max}\,\{i+j\,|\,  H_{\mm}^{i}(M)_{j}\neq 0\}.
\end{align*}

\begin{Remark}
 Let $a_i(M): =\max\,\{j\,|\, H_{\mm}^{i}(M)_{j}\neq 0\}$ with the convention that
$a_i(M):=-\infty$ if $H_{\mm}^{i}(M)=0$. Then
\begin{align*}
\depth\,(M)&=\min\,\{i\,|\, a_i(M)\neq -\infty\},\\
\reg\,(M)&=\max\,\{a_i(M)+i \mid i\geq 0\}.
\end{align*}
\end{Remark}

When $S$ is  a polynomial ring over a field $\KK$ and  $I$ is a nonzero proper homogeneous ideal in $S$,
 depth and  regularity of $I$ are closely related to the minimal free resolution and graded Betti numbers of $I$ in the following way. Suppose that $I$
admits the following minimal free resolution

\vspace{1mm}
$$0\rightarrow \bigoplus\limits_{j}S(-j)^{\beta_{p,j}(I)}\rightarrow \bigoplus\limits_{j}S(-j)^{\beta_{p-1,j}(I)}\rightarrow \cdots\rightarrow \bigoplus\limits_{j}S(-j)^{\beta_{0,j}(I)}\rightarrow I\rightarrow 0,$$
where  $S(-j)$ is an $S$-module obtained by shifting
the degrees of $S$ by $j$. The $(i,j)$-th graded Betti number $\beta_{i,j}(I)$  is
an invariant of $I$ that equals the number of minimal generators of degree $j$ in the
$i$-th syzygy module of $I$.

Let $\mbox{pd}\,(I)$ denote the {\it projective dimension} of $I$. Then
\begin{align*}
\mbox{pd}\,(I)&=\mbox{max}\,\{i\ |\ \beta_{i,j}(I)\neq 0 \text{\ for some }j\,\},\\
\reg\,(I)&=\mbox{max}\,\{j-i\ |\ \beta_{i,j}(I)\neq 0\},\\
\depth\,(I)&=n-\mbox{pd}\,(I).
\end{align*}

By looking at the minimal free resolution and Auslander-Buchsbaum formula (see Theorem 1.3.3 of \cite{BH}), it is easy to see that
\begin{Lemma}{\em (\cite[Lemma 1.3]{HTT})}
\label{quotient}
Let $u\in S$ be a monomial of degree $d$ and $J=(u)$, and let $I\subset S$ be a proper nonzero homogeneous ideal. Then
\begin{itemize}
\item[(1)] $\reg\,(I)=\reg\,(S/I)+1$;
\item[(2)] $\reg\,(S/J)=d-1$.
\end{itemize}
\end{Lemma}

The following lemmas are often used for computing  regularity and depth  of a module.

\begin{Lemma}
	\label{product}{\em (\cite[Theorem 2.5]{CH})}
$S=\KK[x_{1},\dots,x_{n}]$ be a polynomial ring and let $I\subset S$ be a graded ideal with $\dim\,(R/I)\leq1$. Then, for any finitely generated graded $S$-module $M$, we have $\reg\,(IM)\leq\reg\,(I)+\reg\,(M)$.
\end{Lemma}

\begin{Lemma}  {\em (\cite[Lemmas 2.1 and 3.1]{HT})}
	\label{exact}
Let $0\longrightarrow M\longrightarrow N\longrightarrow P\longrightarrow 0$ be a short exact
	sequence of finitely generated graded S-modules. Then we have
	\begin{itemize}
		\item[(1)]$\reg\,(N)\leq max\{\reg\,(M), \reg\,(P)\}$, the equality holds if $\reg\,(P) \neq \reg\,(M)-1$.
		\item[(2)] $\reg\,(M)\leq max\{\reg\,(N), \reg\,(P)+1\}$, the equality holds if $\reg\,(N) \neq \reg\,(P)$.
		\item[(3)]$\depth\,(N)\geq min\{\depth\,(M), \depth\,(P)\}$, the equality holds if $\depth\,(P) \neq \depth\,(M)-1$.
		\item[(4)]$\depth\,(M)\geq min\{\depth\,(N), \depth\,(P)+1\}$, the equality holds if $\depth\,(N)\\ \neq \depth\,(P)$.
	\end{itemize}
\end{Lemma}

\begin{Lemma}{\em (\cite[Lemma 2.2, Lemma 3.2]{HT})}
\label{sum1}
Let $S_{1}=\KK[x_{1},\dots,x_{m}]$ and $S_{2}=\KK[x_{m+1},\dots,x_{n}]$ be two polynomial rings  over $\KK$,  $I\subset S_{1}$ and
$J\subset S_{2}$ be two nonzero homogeneous  ideals. Let $S=S_1\otimes_\KK S_2$.  Then we have
\begin{itemize}
\item[(1)] $\reg\,(S/(I+J))=\reg\,(S_1/I)+\reg\,(S_2/J)$;
\item[(2)]$\depth\,(S/(I+J))=\depth\,(S_1/I)+\depth\,(S_2/J)$;
\item[(3)]$\reg\,(S/JI)=\reg\,(S_1/I)+\reg\,(S_2/J)+1$;
\item[(4)]$\depth\,(S/JI)=\depth\,(S_1/I)+\depth\,(S_2/J)+1$.
\end{itemize}
In particular,  if  $u$ is a monomial of degree  $d$ such that $\mbox{supp}\,(u)\cap \mbox{supp}\,(I)=\emptyset$, let $J=(u)$,  then  $\mbox{reg}\,(J)=d$ and $\mbox{reg}\,(JI)=\mbox{reg}\,(I)+d$.
\end{Lemma}


Let $I=Q_1\cap\cdots\cap Q_m$ be a primary decomposition of the ideal $I$. For $P\in \Ass(S/I)$,
we denote $Q_{\subseteq P}$ to be the intersection of all $Q_i$ with $\sqrt{Q_i}\subseteq P$.

\begin{Lemma}{\em (\cite[Theorem 3.7]{CEHH})}
\label{symbolic powers}
The $k^{th}$ symbolic power of a monomial ideal $I$ is
\[
I^{(k)}= \bigcap\limits_{P \in \Ass(I)}Q_{\subseteq P}^k.
\]
\end{Lemma}

For a positive integer $n$,  we set $[n]=\{1,2,\ldots, n\}$.
\begin{Lemma}{\em (\cite[Theorem 5.6]{HNT})}
	\label{sum2}
Let $\KK$ be a field of  $\Char\,(\KK)=0$, and let $S_{1}=\KK[x_{1},\dots,x_{m}]$, $S_{2}=\KK[x_{m+1},\dots,x_{n}]$ be two polynomial rings over  $\KK$,  and $I\subset S_{1}$,
		$J\subset S_{2}$ be two nonzero monomial ideals. Let $S=S_1\otimes_\KK S_2$.
		Then for any $t\geq 1$,  we have
\begin{itemize}
	\item[(1)]$\reg\left(\frac{S}{(I+J)^{(t)}}\right)=\max\limits_{\begin{subarray}{c}
			i\in [t-1]\\
			j\in [t]\\
	\end{subarray}} \{\reg\left(\frac{S_1}{I^{(t-i)}}\right)+\reg\left(\frac{S_2}{J^{(i)}}\right)+1,\reg\left(\frac{S_1}{I^{(t-j+1)}}\right)+\reg\left(\frac{S_2}{J^{(j)}}\right)\}$,
	\item[(2)]$\!\!\depth\!\left(\frac{S}{(I+J)^{(t)}}\right)\!=\!\min\limits_{\begin{subarray}{c}
			i\in [t-1]\\
			j\in [t]\\
	\end{subarray}}\! \{\depth\!\left(\frac{S_1}{I^{(t-i)}}\right)+\depth\!\left(\frac{S_2}{J^{(i)}}\right)+1, \depth\!\left(\frac{S_1}{I^{(t-j+1)}}\right)+\depth\!\left(\frac{S_2}{J^{(j)}}\right)\}$.
\end{itemize}
Obviously,  we also have $\reg(S/I^{(i)})=\reg(S_1/I^{(i)})$ and $\reg(S/J^{(i)})=\reg(S_2/J^{(i)})$ for any $i\ge 1$.
\end{Lemma}

\section{Primary decomposition  of  edge ideal of  join of  graphs}
In this section, we provide  primary decomposition of the edge ideal of the join of some graphs  in terms of that information of the edge ideal of every weighted oriented graph.
First, we give a definition of the join of   weighted oriented graphs.

\begin{Definition}\label{join}
Let  $r\ge 2$ be an integer and  $D_1,\ldots,D_r$ be weighted oriented graphs over   pairwise  disjoint vertex sets $V_1,\ldots, V_r$, respectively. The join of  $D_1,\ldots,D_r$, denoted by  $D_1*\cdots*D_r$,  is a weighted oriented graph over the vertex set   $V_1\sqcup \cdots \sqcup V_r$,  whose  edge set is
\[
E_1\cup\cdots \cup E_r\cup \{(x, y): x\in V_i, y\in  V_{j}\text {\ with\ }i<j\}.
\]
In particular, if $r=2$, then the  join $D_1*D_2$ of   $D_1$ and $D_2$  is a weighted oriented graph over the vertex set  $V_1\sqcup V_2$ whose  edge set is $E_1\cup E_2\cup \{(x,y) : x\in V_1, y\in  V_2\}$.
\end{Definition}

\begin{Example} \label{exam1}
The following are some typical examples of the join of two and three weighted oriented graphs, respectively.
	
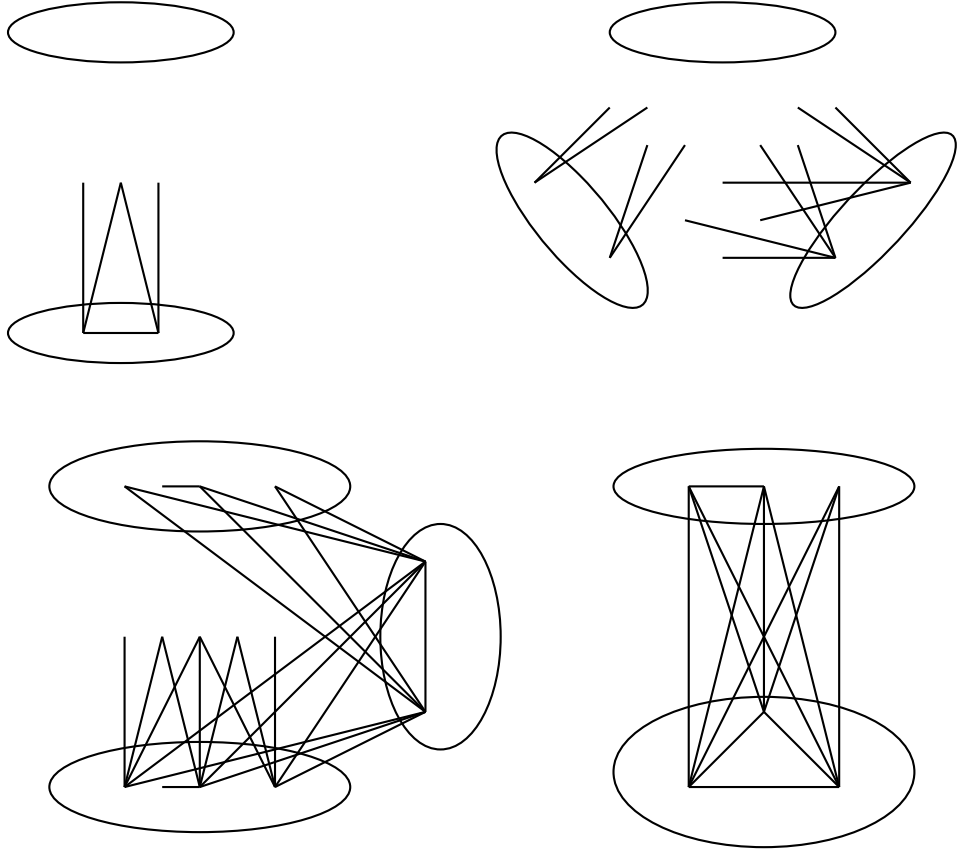
\vspace{0.5cm}
		\begin{center}
			\begin{tikzpicture}[thick,>=stealth]
			
			\setlength{\unitlength}{1mm}
			
			\thicklines

			\pgfpathellipse{\pgfpointxy{4}{6}}{\pgfpointxy{1.5}{0}}{\pgfpointxy{0}{0.4}}
			\pgfusepath{draw},
			
			\pgfpathellipse{\pgfpointxy{2}{3.5}}{\pgfpointxy{-1}{1}}{\pgfpointxy{0.1}{0.6}}
			\pgfusepath{draw}
			
			\pgfpathellipse{\pgfpointxy{6}{3.5}}{\pgfpointxy{1.1}{1}}{\pgfpointxy{0}{0.6}}
			\pgfusepath{draw}

			\pgfpathellipse{\pgfpointxy{-4.0}{2.0}}{\pgfpointxy{1.5}{0}}{\pgfpointxy{0}{0.4}}
			\pgfusepath{draw}
			
			\pgfpathellipse{\pgfpointxy{-4.0}{6.0}}{\pgfpointxy{1.5}{0}}{\pgfpointxy{0}{0.4}}
			\pgfusepath{draw}

			\put(29,60){$x_{11}$}
			\put(35,60){\circle*{1.5}}
			
			\put(47,60){$x_{12}$}
			\put(45,60){\circle*{1.5}}
			
			\put(11,42){$x_{21}$}
			\put(15,40){\circle*{1.5}}
			
			\put(24,27){$x_{22}$}
			\put(25,30){\circle*{1.5}}
			
			\put(52,27){$x_{31}$}
			\put(55,30){\circle*{1.5}}
			
			\put(65,43){$x_{32}$}
			\put(65,40){\circle*{1.5}}
			
			\put(35,60){\vector(-1,-1){11}}
			\draw[solid](2.5,5)--(1.5,4);
			\put(35,60){\vector(-1,-3){6}}
			\draw[solid](3,4.5)--(2.5,3);
			
			\put(45,60){\vector(-3,-2){17}}
			\draw[solid](3,5)--(1.5,4);
			\put(45,60){\vector(-2,-3){15}}
			\draw[solid](3.5,4.5)--(2.5,3);
			
			\put(35,60){\vector(2,-3){11}}
			\draw[solid](4.5,4.5)--(5.5,3);
			\put(35,60){\vector(3,-2){17}}
			\draw[solid](5,5)--(6.5,4);
			
			\put(45,60){\vector(1,-3){6}}
			\draw[solid](5,4.5)--(5.5,3);
			\put(45,60){\vector(1,-1){11}}
			\draw[solid](5.5,5)--(6.5,4);
			
		    \put(15,40){\vector(4,-1){22}}
			\draw[solid](3.5,3.5)--(5.5,3);
			\put(15,40){\vector(1,0){30}}
			\draw[solid](4,4)--(6.5,4);
			
			\put(25,30){\vector(1,0){15}}
			\draw[solid](4,3)--(5.5,3);
			\put(25,30){\vector(4,1){20}}
			\draw[solid](4.5,3.5)--(6.5,4);

			\put(-47,61){$x_{11}$}
			\put(-45,60){\circle*{1.5}}
			
			\put(-37,61){$x_{12}$}
			\put(-35,60){\circle*{1.5}}
			
			\put(-47,18){$x_{21}$}
			\put(-45,20){\circle*{1.5}}
			
			\put(-37,18){$x_{22}$}
			\put(-35,20){\circle*{1.5}}

			\put(-45,20){\vector(1,0){7}}
			\draw[solid](-4.5,2)--(-3.5,2);
			\put(-45,60){\vector(0,-1){20}}
			\draw[solid](-4.5,4)--(-4.5,2);
			\put(-45,60){\vector(1,-4){7}}
			\draw[solid](-4,4)--(-3.5,2);
			
			\put(-35,60){\vector(-1,-4){7}}
			\draw[solid](-4,4)--(-4.5,2);
			\put(-35,60){\vector(0,-1){20}}
			\draw[solid](-3.5,4)--(-3.5,2);
		
			\put(-40,10){$(a)$};
			\put(40,10){$(b)$};

			\end{tikzpicture}
		\end{center}

\vspace{0.8cm}
\begin{center}
	\begin{tikzpicture}[thick,>=stealth]

	\setlength{\unitlength}{1mm}
	
	\thicklines
	
	\pgfpathellipse{\pgfpointxy{-4}{2}}{\pgfpointxy{2}{0}}{\pgfpointxy{0}{0.6}}
	\pgfusepath{draw}
	
	\pgfpathellipse{\pgfpointxy{-0.8}{4}}{\pgfpointxy{0.8}{0}}{\pgfpointxy{0}{1.5}}
	\pgfusepath{draw}
	
	\pgfpathellipse{\pgfpointxy{-4}{6}}{\pgfpointxy{2}{0}}{\pgfpointxy{0}{0.6}}
	\pgfusepath{draw}

	\pgfpathellipse{\pgfpointxy{3.5}{6}}{\pgfpointxy{2}{0}}{\pgfpointxy{0}{0.5}}
	\pgfusepath{draw}

	\pgfpathellipse{\pgfpointxy{3.5}{2.2}}{\pgfpointxy{2}{0}}{\pgfpointxy{0}{1}}
	\pgfusepath{draw}

	\put(-50,62){$x_{11}$}
	\put(-50,60){\circle*{1.5}}
	
	\put(-40,62){$x_{12}$}
	\put(-40,60){\circle*{1.5}}
	
	\put(-30,62){$x_{13}$}
	\put(-30,60){\circle*{1.5}}
	
	\put(-50,17){$x_{21}$}
	\put(-50,20){\circle*{1.5}}
	
	\put(-40,17){$x_{22}$}
	\put(-40,20){\circle*{1.5}}
	
	\put(-30,17){$x_{23}$}
	\put(-30,20){\circle*{1.5}}
	
	\put(-8,50){$x_{31}$}
	\put(-10,50){\circle*{1.5}}
	
	\put(-8,30){$x_{32}$}
	\put(-10,30){\circle*{1.5}}
	
	\put(-10,50){\vector(0,-1){10}}
	\draw[solid](-1,5)--(-1,3);
	\put(-50,60){\vector(0,-1){24}}
	\draw[solid](-5,4)--(-5,2);
	\put(-50,60){\vector(1,-4){8}}
	\draw[solid](-4.5,4)--(-4,2);
	\put(-50,60){\vector(1,-2){15}}
	\draw[solid](-4,4)--(-3,2);
	\put(-50,60){\vector(1,0){7}}
	\draw[solid](-4.5,6)--(-4,6);
	\put(-50,60){\vector(4,-1){23}}
	\draw[solid](-5,6)--(-1,5);
	\put(-50,60){\vector(4,-3){23}}
	\draw[solid](-5,6)--(-1,3);
	\put(-50,20){\vector(4,1){33}}
	\draw[solid](-5,2)--(-1,3);
	\put(-50,20){\vector(4,3){23}}
	\draw[solid](-5,2)--(-1,5);
	
	\put(-40,60){\vector(-1,-4){7}}
	\draw[solid](-4.5,4)--(-5,2);
	\put(-40,60){\vector(0,-1){26}}
	\draw[solid](-4,4)--(-4,2);
	\put(-40,60){\vector(1,-4){7}}
	\draw[solid](-3.5,4)--(-3,2);
	\put(-50,20){\vector(1,0){7}}
	\draw[solid](-4.5,2)--(-4,2);
	\put(-40,60){\vector(3,-1){20}}
	\draw[solid](-4,6)--(-1,5);
	\put(-40,60){\vector(1,-1){13}}
	\draw[solid](-4,6)--(-1,3);
	\put(-40,20){\vector(3,1){20}}
	\draw[solid](-4,2)--(-1,3);
	\put(-40,20){\vector(1,1){16}}
	\draw[solid](-4,2)--(-1,5);
	
	\put(-30,60){\vector(-1,-2){15}}
	\draw[solid](-4,4)--(-5,2);
	\put(-30,60){\vector(-1,-4){8}}
	\draw[solid](-3.5,4)--(-4,2);
	\put(-30,60){\vector(0,-1){24}}
	\draw[solid](-3,4)--(-3,2);
	\put(-30,60){\vector(2,-1){10}}
	\draw[solid](-3,6)--(-1,5);
	\put(-30,60){\vector(2,-3){8.5}}
	\draw[solid](-3,6)--(-1,3);
	\put(-30,20){\vector(2,1){10}}
	\draw[solid](-3,2)--(-1,3);
	\put(-30,20){\vector(2,3){8.5}}
	\draw[solid](-3,2)--(-1,5);
	
	\put(25,62){$x_{11}$}
	\put(25,60){\circle*{1.5}}
	
	\put(35,62){$x_{12}$}
	\put(35,60){\circle*{1.5}}
	
	\put(45,62){$x_{13}$}
	\put(45,60){\circle*{1.5}}
	
	\put(25,17){$x_{21}$}
	\put(25,20){\circle*{1.5}}
	
	\put(33,25){$x_{22}$}
	\put(35,30){\circle*{1.5}}
	\put(44,17){$x_{23}$}
	\put(45,20){\circle*{1.5}}
	
	\put(25,60){\vector(1,0){7}}
	\draw[solid](2.5,6)--(3.5,6);
    \put(25,20){\vector(1,0){12}}
    \draw[solid](2.5,2)--(4.5,2);
	\put(25,20){\vector(1,1){7}}
	\draw[solid](2.5,2)--(3.5,3);
    \put(35,30){\vector(1,-1){7}}
    \draw[solid](3.5,3)--(4.5,2);
    \put(25,60){\vector(0,-1){25}}
    \draw[solid](2.5,6)--(2.5,2);
    \put(25,60){\vector(1,-3){7}}
    \draw[solid](2.5,6)--(3.5,3);
    \put(25,60){\vector(1,-2){16}}
    \draw[solid](2.5,6)--(4.5,2);
    \put(35,60){\vector(0,-1){15}}
    \draw[solid](3.5,6)--(3.5,3);
    \put(35,60){\vector(-1,-4){6}}
    \draw[solid](3.5,6)--(2.5,2);
	\put(35,60){\vector(1,-4){6}}
	\draw[solid](3.5,6)--(4.5,2);
	\put(45,60){\vector(-1,-2){16}}
	\draw[solid](4.5,6)--(2.5,2);
	\put(45,60){\vector(-1,-3){7}}
	\draw[solid](4.5,6)--(3.5,3);
	\put(45,60){\vector(0,-1){20}}
	\draw[solid](4.5,6)--(4.5,2);

	\put(-40,0){$(c)$};
	\put(40,0){$(d)$};

	
	\end{tikzpicture}
\end{center}
\end{Example}

\vspace{1.5cm}
$Fig.1$ Some typical examples of some classes of weighted bipartite and $3$-partite graphs

\medskip

\begin{Remark}\label{edgeideal}
	Let D:=$D_1*\cdots*D_r$ be the join  of   weighted oriented graphs  $D_1,\ldots,D_r$, where the  vertex set of $D_i$ is  $V_i=\{x_{ij}\mid  j\in [n_i]\}$ and  $n_i=|V_i|$ for any  $i\in [r]$.
Let $S=\KK[x_{ij}: i\in [r], j\in [n_i]]$.
Then the  edge ideal of $D$ is a monomial ideal of $S$
 \[
 I(D)=\sum\limits_{i=1}^{r}I(D_i)+(x_{ij}x_{k\ell}^{w_{k\ell}}\mid  j\in [n_i], \ell\in[n_{k}]\text {\ where\ }1\le i<k\le r).
 \]
In particular, if each $D_i$ consists of isolated vertices, then  \[
 I(D)=(x_{ij}x_{k\ell}^{w_{k\ell}}\mid  j\in [n_i], \ell\in[n_{k}]\text {\ where\ }1\le i<k\le r).
 \]
\end{Remark}

\medskip
A monomial ideal is called {\em irreducible} if it cannot be written as proper intersection of two other monomial ideals.
It is called {\em reducible} if it is not provide. It is well known that a monomial ideal is irreducible if and only if it is generated
by pure powers of variables, that is, it  has  the form $(x_{i_1}^{a_1},\ldots,x_{i_k}^{a_k})$. The  following lemma is a fundamental fact.

\begin{Lemma} \label{decomposition}{\em(\cite[Theorem 1.3.1]{HH})}
Let $I\subset S$ be a monomial ideal. Then there exists a unique decomposition
$$I= Q_1\cap\cdots\cap Q_r$$
such that  none of the  $Q_i$  can be omitted in this intersection and each  $Q_i$ is an irreducible monomial ideal.
\end{Lemma}

This decomposition is called  {\it irredundant presentation}  of $I$ and each $Q_i$ is called an {\it irreducible
component} of $I$.

\medskip
 For a monomial ideal $I\subset S$,  we denote by $\mathcal{G}(I)$  the unique minimal set of monomial generators of $I$. Let  $u\in S$ be a monomial, we set $\supp\,(u)=\{x_i: x_i|u\}$. If $\mathcal{G}(I)=\{u_1,\ldots,u_m\}$, we set $\supp\,(I)=\bigcup\limits_{i=1}^{m}\supp\,(u_i)$.

Now, we determine the irredundant presentation  of the edge ideal $ I(D_1*\cdots*D_r)$ by using the irredundant presentation  of the edge ideal $I(D_i)$
of each  weighted oriented graph  $D_i$.

\begin{Theorem}\label{rdecompose}
	Let $r\geq2$ be an integer, and let $D:=D_1*\cdots*D_r$ be the join  of   weighted oriented graphs  $D_1,\ldots,D_r$, where the  vertex set of $D_i$ is  $V_i=\{x_{ij}\,|\,  j\in [n_i]\}$ and  $n_i=|V_i|$ for $i\in [r]$.
Set  $I_i= (x_{ij}\,|\,j\in[n_i])$,  $I_i^{w_i}=(x_{ij}^{w_{ij}}\,|\,j\in[n_i])$,  $\widehat{I}_i=\sum\limits_{j=1}^{i-1}{I_j}+\sum\limits_{j=i+1}^{r}{I_j^{w_j}}$ and   $\sum\limits_{j=1}^{0}{I_j}=\sum\limits_{j=r+1}^{r}{I_j^{w_j}}=(0)$ by convention. Let $I(D_i)=Q_{i1}\cap \dots \cap Q_{it_i}$ be the  irredundant presentation  of $I(D_i)$. Then the  irredundant presentation of $I(D)$ is
\[
I(D)=\bigcap\limits_{i=1}^r\bigcap\limits_{j=1}^{t_i}{(\widehat{I}_i+Q_{ij})}.
\]
\end{Theorem}
\begin{proof}
We first prove  that $I(D)=\bigcap\limits_{i=1}^r{(\widehat{I}_i+I(D_i))}$ by induction on $r$. If $r=2$, then
	\begin{align*}
	(\widehat{I}_1+I(D_1))\cap (\widehat{I}_2+I(D_2))&= \widehat I_1\cap \widehat I_2+\widehat I_1\cap I(D_2)+\widehat I_2\cap I(D_1)+I(D_1)\cap I(D_2)\\&= \widehat I_1\widehat I_2+I(D_1)+I(D_2)\\
	&=(x_{1i}x_{2j}^{w_{2j}}\mid i\in [n_1], j\in[n_{2}])+I(D_1)+I(D_2)=I(D)
	\end{align*}
where the second  equality holds because of  $\supp\,(\widehat{I}_1)\cap \supp\,(\widehat{I}_2)=\emptyset$, $I(D_1)\subset \widehat I_2$  and  $I(D_2)\subset \widehat I_1$.

 Suppose that $r\ge 3$ and that the statement holds for $r-1$, that is
	\begin{align*}
		I(D_1*\cdots*D_{r-1})=&\bigcap\limits_{i=1}^{r-1}(\widehat{I}_i+I(D_i))
		=(I_2^{w_2}+\cdots+ I_{r-1}^{w_{r-1}}+I(D_1))\cap(I_1+I_3^{w_3}+\cdots\\
		&+I_{r-1}^{w_{r-1}}+I(D_2))\cap\cdots\cap(I_1+\cdots+I_{r-2}+I(D_{r-1})).
	\end{align*}
Thus
 \begin{align*}
\bigcap\limits_{i=1}^r{(\widehat{I}_i+I(D_i))}
&=\bigcap\limits_{i=1}^{r-1}{(\widehat{I}_i+I(D_i))}\cap (\widehat{I}_r+I(D_r))\\
	&=(I_2^{w_2}+\cdots+I_{r}^{w_{r}}+I(D_1))\cap(I_1+I_3^{w_3}+\cdots+I_{r}^{w_{r}}+I(D_2))\\
&\cap\cdots\cap(I_1+\cdots+I_{r-2}+I_{r}^{w_{r}}+I(D_{r-1}))\cap(I_1+\cdots+I_{r-1}+I(D_r))\\
&=(I_r^{w_r}+I(D_1*\cdots*D_{r-1}))\cap(I_1+\cdots+I_{r-1}+I(D_r))\\
	&=\sum\limits_{i=1}^{r-1}I_i\cap I_r^{w_r}+I_r^{w_r}\cap I(D_r)+I(D_1*\cdots*D_{r-1})\cap (I_1+\cdots+I_{r-1})\\
&+I(D_1*\cdots*D_{r-1})\cap I(D_r)\\
	&=\sum\limits_{i=1}^{r-1}I_iI_r^{w_r}+I(D_1*\cdots*D_{r-1})+I(D_r)=I(D)
\end{align*}
where the penultimate  equality holds because of $\supp\,(I_i)\cap \supp\,(I_r^{w_r})=\emptyset$,  $I(D_r)\subset I_r^{w_r}$  and  $I(D_1*\cdots*D_{r-1})\subset (I_1+\cdots+I_{r-1})$, the last
 equality holds because $I(D_1*\cdots*D_{r-1})=\sum\limits_{i=1}^{r-1}I(D_i)+(x_{ij}x_{k\ell}^{w_{k\ell}}\mid  j\in [n_i], \ell\in[n_{k}]\text {\ where\ }1\le i<k\le r-1)$.

Since each $I(D_i)=\bigcap\limits_{j=1}^{t_i}Q_{ij}$, one has $\widehat I_i+I(D_i)=\widehat I_i+\bigcap\limits_{j=1}^{t_i}Q_{ij}=\bigcap\limits_{j=1}^{t_i}(\widehat{I}_i+Q_{ij})$. Note that  $\supp\,(\widehat{I}_i)\cap \supp\,(Q_{ij})=\emptyset$, it implies that the ideal $\widehat{I}_i+Q_{ij} $ is  irreducible   for any $i\in[r]$, $j\in[n_i]$.
It follows that
\[
I(D)=\bigcap\limits_{i=1}^r{(\widehat{I}_i+I(D_i))}=\bigcap\limits_{i=1}^r\bigcap\limits_{j=1}^{t_i}{(\widehat{I}_i+Q_{ij})}.
\]
We finish the proof.
\end{proof}

\medskip
Given an ideal $I\subset S$, we set
 $$\bight\,(I)=\sup\{\Ht\,(P)\mid  P \text{ is a   minimal prime  ideal  of } S \text{ over }  I\}.$$

As a direct consequence of the above theorem, one has the following corollary.
\begin{Corollary}\label{height}
Let  $D:=D_1*\cdots*D_r$ be the join  of   weighted oriented graphs  $D_1,\ldots,D_r$ as in Theorem \ref{rdecompose}.
	Then
\begin{enumerate}
 \item[(a)] $\Ht\,(I(D))=\min\{\sum\limits_{j=1, j\ne i}^{r}n_j+\Ht\,(I(D_i))\mid i\in [r]\}$,
 \item[(b)] $\dim\,(S/I(D))=\max\{n_i-\Ht\,(I(D_i))\mid i\in [r]\}$,
 \item[(c)] $\bight\,(I(D))=\max\{\sum\limits_{j=1,j\ne i}^{r}n_j+\Ht\,(I(D_i))\mid i\in [r]\}$.
\end{enumerate}
 \end{Corollary}

\begin{Corollary}\label{unmixed}
Let  $D:=D_1*\cdots*D_r$ be the join  of   weighted oriented graphs  $D_1,\ldots,D_r$ as Theorem \ref{rdecompose}.
	Then
$I(D)$ is unmixed if and only if each $I(D_i)$ is unimxed and $n_j-\height\,(I(D_j))=n_k-\height\,(I(D_k))$ for any
different $j,k\in [r]$.
 \end{Corollary}
\begin{proof} By Theorem \ref{rdecompose} and  (a) and (c) in Corollary \ref{height},  we know that
$I(D)$ is unmixed if and only if $\sum\limits_{j=1,j\ne i}^{r}n_j+\Ht\,(Q_{ij})=\sum\limits_{j=1,j\ne i}^{r}n_j+\Ht\,(Q_{ik})$ for  any  different $j,k\in [t_i]$  and $i\in [r]$, and $\sum\limits_{i=1,i\ne j}^{r}n_i+\Ht\,(I(D_j))=\sum\limits_{i=1,i\ne k}^{r}n_i+\Ht\,(I(D_k))$ for any $j,k\in [r]$. This implies that each $I(D_i)$ is unimxed and
$n_j-\Ht\,(I(D_{j}))=n_k-\Ht\,(I(D_{k}))$, as wished.
	\end{proof}		

\begin{Corollary}\label{isolated vertices}
	Let $r\geq2$ be a positive integer and let $D:=D_1*\cdots*D_r$ be the join  of   weighted oriented graphs  $D_1,\ldots,D_r$, where $D_i$ consists of isolated vertices with the  vertex set  $V_i=\{x_{ij}\mid  j\in [n_i]\}$ and  $n_i=|V_i|$ for any  $i\in [r]$.
	Then
$I(D)$ is unmixed if and only if  $n_1=n_2=\cdots=n_r$.
 \end{Corollary}

\medskip
\section{Symbolic powers of the edge ideal of the join of graphs consisting of isolated vertices}
In this section, we study depth and  regularity of symbolic powers and  ordinary powers of the edge ideal of the join of weighted oriented graphs, where each graph  consists of isolated vertices.
For the edge ideal of the join of two weighted oriented graphs, we give the exact formulas for  depth and regularity of their powers and symbolic powers.
For the edge ideal of the join of $r(\ge 3)$ weighted oriented graphs, we provide the exact  depth formulas and also give the  upper bounds on   regularity of  symbolic powers.

 \begin{Lemma}\label{regular}{\em (\cite[Lemma 4.4]{BHT})}
 	Let $u_1,\ldots, u_m$ be a regular sequence of homogeneous polynomials in $S$ with  $\deg\,(u_1)=\cdots=\deg\,(u_1)=d$. Let $I=(u_1,\ldots, u_m)$. Then we have $\reg\,(I^t)=dt+(d-1)(m-1)$  for all $t\ge 1$.
 \end{Lemma}

\medskip
 First, we consider the case $r=2$ and  each $D_i$ consisting of isolated vertices.

\begin{Theorem}\label{bipartitepower}
	Let $D:=D_1*D_2$ be the join  of  two weighted oriented graphs  $D_1$ and $D_2$, where $D_i$ consists of isolated vertices with  $V_i=\{x_{ij}\,|\, j\in [n_i]\}$,  $n_i=|V_i|$ for all  $i\in [2]$. Then, for all $t\ge 1$, we have
\begin{itemize}
		\item[(1)] $\depth\,(S/I(D)^t)=1$;
		\item[(2)] $\reg\,(S/I(D)^t)=\sum\limits_{x\in V}{w(x)}-|V|+1+(t-1)(w+1)$, where $w=max\{w(x) \mid  x\in V\}$ and  $V=V_1\sqcup V_2$.
	\end{itemize}
	\end{Theorem}
\begin{proof} Note that $I(D)^{t}=I^tJ^t$, where $I=(x_{11},\ldots, x_{1n_1})$, $J=(x_{21}^{w_{21}},\ldots, x_{2n_2}^{w_{2n_2}})$ and $\supp\,(I)\cap \supp\,(J)=\emptyset$.
Then, by Lemma \ref{sum1} (3) and (4), one has
\[
\depth\,(S/I(D)^t)=\depth\,(S/J^tI^t)=\depth\,(S_1/I^t)+\depth\,(S_2/J^t)+1
\]
and
\[
\reg\,(S/I(D)^t)=\reg\,(S/J^tI^t)=\reg\,(S_1/I^t)+\reg\,(S_2/J^t)+1
\]
Therefore, we obtain $\depth\,(S/I(D)^t)=1$, since both $I$ and $J$ are  $\mathfrak{m}_1$-primary and $\mathfrak{m}_2$-primary  in $S_1$ and $S_2$ respectively, where   $\mathfrak{m}_1=(x_{11},\ldots, x_{1n_1})$ and $\mathfrak{m}_2=(x_{21},\ldots, x_{2n_2})$. By Lemmas \ref{quotient} (1)  and  \ref{regular} and the following lemma, we obtain
\[
\reg\,(S/I(D)^t)=t+\reg\,(S_2/J^t)=\sum\limits_{x\in V}{w(x)}-|V|+1+(t-1)(w+1),
\]  where $w=max\{w(x) \mid  x\in V\}$ .
\end{proof}		

\begin{Lemma}\label{colon}
	Let  $J=(x_{21}^{w_{21}},\ldots, x_{2n_2}^{w_{2n_2}})$. Then, for any $t\ge 1$, we have
	\[
	\reg\,(S/J^t)=\sum\limits_{i=1}^{n_2}w_{2i}-n_2+(t-1)w
	\]
	where $w=max\{w_{2i} \mid  i \in [n_2]\}$.
\end{Lemma}
\begin{proof}
	We apply induction on $n_2$ and $t$. Without loss of generality, we can assume that $w=w_{2n_2}$. The case $n_{2}=1$ or $t=1$ is  obvious by Lemmas \ref{quotient} (2) and
 \ref{sum1} (1).
	
	Now, we assume that  $n_2, t\geq 2$. By some simple calculations, we obtain that
 $J^{t}:x_{21}^{w_{21}}=J^{t-1}$ and $(J^t,x_{21}^{w_{21}})=(x_{22}^{w_{22}},\ldots, x_{2n_2}^{w_{2n_2}})^t+(x_{21}^{w_{21}})$.
	By Lemma \ref{sum1} (1) and  inductive hypothesis, we obtain
	\begin{eqnarray*}
	\reg\,(S/(J^t,x_{21}^{w_{21}}))&=&\reg\,(S/(x_{22}^{w_{22}},\ldots, x_{2n_2}^{w_{2n_2}})^t)+{w_{21}}-1\\
	&=&\sum\limits_{i=2}^{n_2}w_{2i}-(n_2-1)+(t-1)w+w_{21}-1\\
	&=&\sum\limits_{i=1}^{n_2}w_{2i}-n_2+(t-1)w
\end{eqnarray*}
and
	\[
	\reg\,(S/(J^t:x_{21}^{w_{21}}))=\reg\,(S/J^{t-1})=\sum\limits_{i=1}^{n_2}w_{2i}-n_2+(t-2)w.
	\]
The desired result  holds by Lemma \ref{exact} (1) and 	the following exact sequence
\[
0\longrightarrow (S/J^t:x_{21}^{w_{21}})(-w_{21})\longrightarrow S/J^t\longrightarrow S/(J^t,x_{21}^{w_{21}})\longrightarrow 0. \qedhere
\]
\end{proof}		

Next,  we consider the case $r\ge 3$ and  each $D_i$ consisting of isolated vertices.
\begin{Lemma}\label{decompose}
	Let $r\geq3$ be an integer, and let $D:=D_1*\cdots*D_r$ be the join  of   weighted oriented graphs  $D_1,\ldots,D_r$, where $D_i$ consists of isolated vertices with  $V_i=\{x_{ij}\,|\,  j\in [n_i]\}$ and  $n_i=|V_i|$ for any  $i\in [r]$.  Assume that $I_i= (x_{ij}\,|\,j\in[n_i])$,  $I_i^{w_i}=(x_{ij}^{w_{ij}}\,|\,j\in[n_i])$ and $\widehat{I}_i=\sum\limits_{j=1}^{i-1}{I_j}+\sum\limits_{j=i+1}^{r}{I_j^{w_j}}$.
	Then
	\begin{itemize}
		\item[(1)] $I(D_2*\cdots*D_r)=\sum\limits_{i=2}^{r-1}\sum\limits_{j=i+1}^{r}{I_iI_j^{w_j}}$;
		\item[(2)] Let $L=\bigcap\limits_{i=2}^{r}\widehat{I}_i$, then  $L=I_1+I(D_2*\cdots*D_r)$.
	\end{itemize}	
\end{Lemma}
\begin{proof}  (1) follows  from Remark \ref{edgeideal}.

	(2) We apply induction on  $r$. If $r=3$, one has
	$L=\widehat{I}_2\cap \widehat{I}_3= (I_1+I_3^{w_3})\cap(I_1+I_2)=I_1+I_2I_3^{w_3}$, as desired.
	If $r\ge 4$, then, by induction, we have
	\begin{align*}
		\bigcap\limits_{i=2}^{r-1}\widehat{I}_i
		=&(I_1+I_3^{w_3}+\cdots+ I_{r-1}^{w_{r-1}})\cap(I_1+I_2+I_4^{w_4}+\cdots+I_{r-1}^{w_{r-1}})\cap\cdots\cap(I_1+\cdots+I_{r-2})\\
		=&I_1+\sum\limits_{i=2}^{r-2}\sum\limits_{j=i+1}^{r-1}{I_iI_j^{w_j}}.
	\end{align*}
 By Theorem   \ref{rdecompose} in the case of all $I(D_i)=0$,  we obtain
	\begin{align*}
	L&=\bigcap\limits_{i=2}^{r}\widehat{I}_i=(I_1+I_3^{w_3}+\cdots+ I_{r}^{w_{r}})\cap(I_1+I_2+I_4^{w_4}+\cdots+I_{r}^{w_{r}})\cap\cdots\cap(I_1+\cdots+I_{r-1})\\
	&=I_1+(I_3^{w_3}+\cdots+ I_{r}^{w_{r}})\cap(I_2+I_4^{w_4}+\cdots+I_{r}^{w_{r}})\cap\cdots\cap(I_2+\cdots+I_{r-1})\\
	&=I_1+I(D_2*\cdots*D_r). 
\end{align*}
\end{proof}	

\begin{Theorem}\label{regularity}
	Let  $r\geq3$ be an integer, and let  $D:=D_1*\cdots*D_r$ be the join  of   weighted oriented graphs  $D_1,\ldots,D_r$ as in Lemma \ref{decompose}.
	Then
	\begin{itemize}
		\item[(1)] 		$\depth\,(S/I(D))=1$;
		\item[(2)] $\reg\,(S/I(D))=\sum\limits_{x\in V}{w(x)}-|V|+1$, where $V=V_1\sqcup\cdots\sqcup V_r$.
	\end{itemize}
\end{Theorem}
\begin{proof} We apply induction on  $r$  with the  $r=2$  case verified in  Theorem \ref{bipartitepower}. Now, assume that $r\geq 3$. By Theorem   \ref{rdecompose} in the case of all $I(D_i)=0$,  one has
$I(D)=\bigcap\limits_{i=1}^r{\widehat I_i}=K\cap L$, where  $K=\widehat I_1=\sum\limits_{j=2}^{r}{I_j^{w_j}}$ and $L=\bigcap\limits_{i=2}^{r}\widehat{I}_i=I_1+I(D_2*\cdots*D_r)$ by Lemma \ref{decompose} (2).
Note that $\supp\,(I_1)\cap \supp\,(I(D_2*\cdots*D_r))=\emptyset$, thus, by Lemma \ref{sum1} (1), (2) and the induction,
 we have
\begin{align*}
\depth\,(S/L)&=\depth\,(S_1/I_1)+\depth\,(S_2/I(D_2*\cdots*D_r))=0+1=1,\\
\reg\,(S/L)&=\reg\,(S_1/I_1)+\reg\,(S_2/I(D_2*\cdots*D_r))=\sum\limits_{x\in V\backslash V_1}{w(x)}-| V\backslash V_1|+1\\
&=\sum\limits_{x\in V}{w(x)}-|V|+1
 \end{align*}	
where the last equality holds because the
weight of each vertex in $V_1$ is one, and
 $S_1=k[x_{11}, \ldots,x_{1n_1}]$ and $S_2=k[x_{21},\ldots,x_{2n_2},\ldots,x_{r1},\ldots, x_{rn_r}]$.
Moreover, we also have
\begin{align*}
	\depth\,(S/K)&=n_1+\depth\,(S_2/\sum\limits_{j=2}^{r}{I_j^{w_j}})=n_1,\\
	\reg\,(S/K)&=\reg\,(K)-1=\sum\limits_{i=2}^ r{\reg\,(I_i^{w_i})}-(r-1)\\
	&=\sum\limits_{i=2}^ r(\sum\limits_{j=1}^{ n_i}w_{ij}-| V_i|+1)-(r-1)=\sum\limits_{x\in V\backslash{V_1}}{w(x)}-| V\backslash{V_1}|\\
	&=\sum\limits_{x\in V}{w(x)}-|V|.
\end{align*}
Notice that $K+L=\sum\limits_{j=2}^{r}{I_j^{w_j}}+I_1+I(D_2*\cdots*D_r)=I_1+\sum\limits_{j=2}^{r}{I_j^{w_j}}$, it follows that  $\depth\,(S/(K,L))=0$
 and $\reg\,(S/(K,L))=\sum\limits_{x\in V}{w(x)}-|V|$.
The desired result holds by Lemma \ref{exact} (2), (4) and the following exact sequence
\[
0\longrightarrow S/K\cap L\longrightarrow S/K\oplus S/L\longrightarrow S/(K,L)\longrightarrow 0.\qedhere
\]
\end{proof}

Let $G=(V,E)$ be a simple graph. If $V=\bigsqcup\limits_{\ell=1}^{r}V_i$ is the disjoint union of $r$ subsets ($V_\ell$ is called the $\ell^{th}$ part, and we set $a_\ell=|V_\ell|$), and if
\[
E=\{\{x,y\}\mid x \text{\ and\ } y  \text{\ are not in the same part}\}.
\]
then $G$ is called a complete $r$-partite graph, denoted by $K_{a_1,a_2,\ldots,a_r}$.
In particular, if $a_1=a_2=\cdots =a_r=1$, then  $G$ is called a complete graph.

\medskip
An immediate consequence of the above theorem is the following corollary.
\begin{Corollary}\label{rpartitepower2}
Let $G=K_{a_1,a_2,\ldots,a_r}$ {\em (resp. $G=K_{n}$)} be  a complete $r$-partite graph {\em (resp. a complete  graph)}, and let $I(G)$  the edge ideal of $G$, then
\[
\depth\,(S/I(G))=\reg\,(S/I(G))=1.
\]
In particular,  $I(G)$ has a linear resolution.
\end{Corollary}
\begin{proof}
This is a direct consequence of the above theorem in the case of  all vertices of $D_i$ having trivial weights.
\end{proof}

In the following,  we provide  a technical lemma, which is useful to prove our main result of this section.

\begin{Lemma}\label{symbolic sum}
	Let $\KK$ be a field of  $\Char\,(\KK)=0$, and let $S=\KK[x_{1},\dots,x_{m},y_1,\dots,y_{n}]$ be a polynomial ring over $\KK$ and $I\subset (x_{1},\dots,x_{m})^2$  a monomial ideal in $\KK[x_{1},\dots,x_{m}]$. Then, for all $t\ge 1$, we have
	\[
	\reg\,((I+(y_1^{w(y_1)},\dots,y_{n}^{w(y_n)}))^{(t)}+(x_{1},\dots,x_{m}) ^{(t)})\\ \leq\sum\limits_{i=1}^n{w(y_i)}-n+1+(t-1)(w+1)
	\]
	where $w=\max\{w(y_i) \,|\, i\in [n]\}$.
\end{Lemma}
\begin{proof}
	Let $K=(I+(y_1^{w(y_1)},\dots,y_{n}^{w(y_n)}))^{(t)}+(x_{1},\dots,x_{m}) ^{(t)}$. Since $x_i^t,y_j^{tw(y_j)}\in K$ for all $i\in[m]$ and $j\in[n]$, and $K$ is $\mathfrak{m}$-primary, where $\mathfrak{m}=(x_{1},\dots,x_{m},y_1,\dots,y_{n})$ is a maximal graded ideal in $S$. Therefore,
	\[
	\reg\,(K)=\max\{j+1:(S/K)_j\neq0\}.
	\]
	Let $A=\sum\limits_{i=1}^n{w(y_i)}-n+1+(t-1)(w+1)$. Without loss of generality, we can assume that $w=w(y_n)$. In order to get the desired assertions, it is sufficient to prove that $(S/K)_j=0$ for all $j\geq A$.
	That is, it is enough to prove that  $f\in K$ for any monomial $f$ with  $\deg\,(f)\geq A$.
	
	Indeed, let $f=(\prod\limits_{i=1}^mx_i^{a_i})(\prod\limits_{j=1}^ny_j^{b_j})$ with $\sum\limits_{i=1}^ma_i+\sum\limits_{j=1}^nb_j\ge A$. We consider the following two cases:
	
	(1) If $\sum\limits_{i=1}^ma_i\geq t$, then $f\in (x_{1},\dots,x_{m})^t=(x_{1},\dots,x_{m})^{(t)}\subseteq K$.	
	
	(2) If $\sum\limits_{i=1}^ma_i\leq t-1$, then $\sum\limits_{j=1}^nb_j\geq A-t+1=\sum\limits_{i=1}^n{w(y_i)}-n+1+(t-1)w$. In this case, let $f_1=\prod\limits_{j=1}^ny_j^{b_j}$, then  $\deg\,(f_1)\geq \sum\limits_{i=1}^n{w(y_i)}-n+1+(t-1)w$.
	
	Next, we prove that $f_1\in (y_1^{w(y_1)},\dots,y_{n}^{w(y_n)})^t$ by induction on  $t$ and $n$, thus  $f_1\in K$, since $(y_1^{w(y_1)},\dots,y_{n}^{w(y_n)})^t\subseteq (I+(y_1^{w(y_1)},\dots,y_{n}^{w(y_n)}))^{(t)}$.
	
	The case  $n=1$ is trivial. If $t=1$, then $\deg\,(f_1)=\sum\limits_{j=1}^n{b_j}\geq \sum\limits_{i=1}^n{w(y_i)}-n+1$.
	In this case, there exists some $i\in [n]$ such that $b_i\geq w(y_i)$, which implies that $f_1\in  (y_1^{w(y_1)},\dots,y_{n}^{w(y_n)})$, as wished.
	On the contrary, if $ b_j\leq w(y_j)-1$ for any $j\in [n]$, then $\deg\,(f_1)\leq \sum\limits_{i=1}^n(w(y_i)-1)=\sum\limits_{i=1}^nw(y_i)-n$,  contradicting  with the assumption that $\deg\,(f_1)\geq \sum\limits_{i=1}^n{w(y_i)}-n+1$.
	
	Now,  we assume that $n,t\geq 2$. By similar arguments as the case $t=1$, we can obtain that
	there exists some $i\in [n]$ such that $b_i\geq w(y_i)$. In this case, we set $b_i=pw(y_i)+q$ with $p\geq 1$, $0\leq q\leq w(y_i)-1$.  We divide into  the following two cases:
	
	(i) If $p\geq t$, then $f_1\in (y_1^{w(y_1)},\dots,y_{n}^{w(y_n)})^t$.
	
	(ii) If $p\leq t-1$, then
	\begin{align*}
		\deg\,(f_1/y_i^{b_i})& \geq \sum\limits_{i=1}^n{w(y_i)}-n+1+(t-1)w-b_i\\
		& \ge \sum\limits_{i=1}^n{w(y_i)}-n+1+(t-1)w-(pw(y_i)+w(y_i)-1)\\
&= \sum\limits_{i=1}^n{w(y_i)}-n+2+(t-1)w-(p+1)w(y_i)\\
& \ge \sum\limits_{i=1}^{n-1}{w(y_i)}-(n-1)+1+(t-1-p)w
	\end{align*}
	where the last  inequality holds because of $w=\max\{w(y_i) \,|\, i\in [n]\}$ and   $w=w(y_n)$.
Hence, by  induction, we have  $f_1/y_i^{b_i}\in (y_1^{w(y_1)},\dots,\widehat {y_i^{w(y_i)}},\dots,y_n^{w(y_n)})^{t-p}$,
	where $\widehat {y_i^{w(y_i)}}$ denotes the element $y_i^{w(y_i)}$ being omitted
	from $(y_1^{w(y_1)},\dots,y_i^{w(y_i)},\dots,y_n^{w(y_n)})^{t-p}$.
	It follows that $f_1\in (y_1^{w(y_1)},\dots,\widehat{y_i^{w(y_i)}},\dots,y_n^{w(y_n)})^{t-p}(y_i^{w(y_i)})^p\subseteq(y_1^{w(y_1)},\dots,y_{n}^{w(y_n)})^t$.
	And this concludes the proof.
\end{proof}

\begin{Theorem}\label{symbolic isolate1}
	Let $D:=D_1*D_2$ be the join  of two  weighted oriented graphs  $D_1$ and $D_2$, where $D_i$ consists of isolated vertices with   $V_i=\{x_{ij}\,|\,  j\in [n_i]\}$ for   $i=1,2$.
	Then, for all $t\ge 1$, we have
\begin{itemize}
	\item[(1)] 		$\depth\,(S/I(D)^{(t)})=1$;
	\item[(2)]  $\reg\,(S/I(D)^{(t)})=\sum\limits_{x\in V}{w(x)}-|V|+1+(t-1)(w+1)$
\end{itemize}
	 where  $w=\max\{w(x) \,|\, x\in V\}$ and $V=V_1\cup V_2$.
\end{Theorem}
\begin{proof} (1) follows from Theorem \ref{bipartitepower}, since 	$\depth\,(S/I(D)^{(t)})=\depth\,(S/I(D)^{t})$ in this case.
	
	(2) It is clear for the  $t=1$  by Theorem \ref{bipartitepower} (2).
	Now, assume that  $t\geq 2$.  Using   notations of Lemma \ref{decompose}, we have $I(D)^{(t)}=I_1^{(t)}\cap (I_2^{w_2})^{(t)}$ by Lemma \ref{symbolic powers}. Thus, by Lemmas \ref{regular} and \ref{colon}, one has
	\begin{align*}
		\reg\,(S/I_1^{(t)})&=\reg\,(S/I_1^{t})=t-1,\\
		\reg\,(S/(I_2^{w_2})^{(t)})&=\reg\,(S/(I_2^{w_2})^{t})=\sum\limits_{x\in V}{w(x)}-|V|+(t-1)w,\\
		\reg\,(S/(I_1^{(t)}+(I_2^{w_2})^{(t)}))&=\reg\,(S/(I_1^{t}+(I_2^{w_2})^{t}))\\
		&=\sum\limits_{x\in V}{w(x)}-|V|+(t-1)w+t-1\\
		&=\sum\limits_{x\in V}{w(x)}-|V|+(t-1)(w+1).
	\end{align*}
	Hence, the desired results hold by using Lemma \ref{exact} (1) and (4) to the following exact sequence
	\[
	0\longrightarrow S/I(D)^{(t)}\longrightarrow S/I_1^{(t)}\oplus S/(I_2^{w_2})^{(t)}\longrightarrow S/(I_1^{(t)}+(I_2^{w_2})^{(t)})\longrightarrow 0.\qedhere
	\]
\end{proof}

Now we present our main theorem of this section.
\begin{Theorem}\label{symbolic isolate2}
	Let $r\geq2$ be an integer and let   $D:=D_1*\cdots*D_r$ be the join  of   weighted oriented graphs  $D_1,\ldots,D_r$, where $D_i$ consists of isolated vertices with  $V_i=\{x_{ij}\,|\,  j\in [n_i]\}$ for any  $i\in [r]$.
	Then, for any $t\ge 1$, we have
	\begin{itemize}
		\item[(1)] 		$\depth\,(S/I(D)^{(t)})=1$;
		\item[(2)]  $\reg\,(S/I(D)^{(t)})\leq \sum\limits_{x\in V}{w(x)}-|V|+1+(t-1)(w+1)$.
	\end{itemize}
	And these	equalities hold  if $w=\max\{w(x) \,|\, x\in V_2\}$, where  $w=\max\{w(x) \,|\, x\in V\}$ and $V=V_1\sqcup\cdots\sqcup V_r$.
\end{Theorem}
\begin{proof}
	We prove the assertion by induction on $r$ and $t$. The cases $t=1$, or $t\geq 2$ and $r=2$ follows from Theorem \ref{regularity}, or Theorem \ref{symbolic isolate1}, respectively.
	
	Now, assume that  $t\geq 2$ and $r\ge 3$. Using  notations of Lemma \ref{decompose},  we have  $I(D)=\bigcap\limits_{i=1}^r\widehat{I}_i$.
	By similar arguments as  the case $r=2$ in Theorem \ref{symbolic isolate1}, we obtain
	\[
	I(D)^{(t)}=(\bigcap\limits_{i=1}^r\widehat{I}_i)^{(t)}=(\bigcap\limits_{i=1}^{r-1}\widehat{I}_i)^{(t)}\cap(\widehat{I}_r)^{(t)}=(J+I_r^{w_r})^{(t)}\cap (\widehat I_r)^{(t)},
	\]
where $J=I(D_1*\cdots*D_{r-1})$.

Note that $\reg\,(S/\widehat I_r^{\,(t)})=t-1$ and $\reg\,(S/((J+I_r^{w_r})^{(t)}+\widehat I_r^{\,(t)}))\leq \sum\limits_{x\in V_r}{w(x)}-|V_r|+1+(t-1)(w'+1)$  by Lemma \ref{symbolic sum}, where $w'=\max\{w(x) \,|\, x\in V_r\}$.  Meanwhile, we also have $\depth\,(S/\widehat I_r^{\,(t)})=n_r$ and $\depth\,(S/((J+I_r^{w_r})^{(t)}+\widehat I_r^{\,(t)}))=0$, since $(J+I_r^{w_r})^{(t)}+\widehat I_r^{\,(t)}$ is $\mathfrak{m}$-primary  in $S$, which is shown in Lemma \ref{symbolic sum}.
Therefore, 	by Lemmas \ref{sum2}, \ref{colon} and  the induction, we have
\begin{align*}
& \reg\left(\frac{S}{(J+I_r^{w_r})^{(t)}}\right)=\max\limits_{\begin{subarray}{c}
 i\in [t-1]\\
 j\in [t]\\
 \end{subarray}} \left\{{\reg\left(\frac{S}{J^{(t-i)}}\right)+\reg\left(\frac{S}{(I_r^{w_r})^{(i)}}\right)+1,\reg\left(\frac{S}{J^{(t-j+1)}}\right)}\right.\\
&\left.{+\reg\left(\frac{S}{(I_r^{w_r})^{(j)}}\right)}\right\}\\
&=\max\limits_{\begin{subarray}{c}
	j\in [t]\\
	\end{subarray}} \left\{\reg\left(\frac{S}{J^{(t-j+1)}}\right)+\reg\left(\frac{S}{(I_r^{w_r})^{(j)}}\right)\right\}\\
&\leq\max\limits_{\begin{subarray}{c}
	j\in [t]\\
	\end{subarray}} \{\sum\limits_{x\in V\setminus V_r}{w(x)}-|V\setminus V_r|+1+(t-j)(w''+1)+\sum\limits_{x\in V_r}{w(x)}-|V_r|+(j-1)w'\}\\
&=\max\limits_{\begin{subarray}{c}
	j\in [t]\\
	\end{subarray}} \{\sum\limits_{x\in V}{w(x)}-|V|+1+(t-j)w''+(j-1)w'+t-j\}\\
&\leq\sum\limits_{x\in V}{w(x)}-|V|+1+(t-1)(w+1),
\end{align*}
where $w''=\max\{w(x) \,|\, x\in V\setminus V_r\}$ and $V=V_1\sqcup\cdots\sqcup V_r$. And the second equality holds because of $\reg\,(S/(I_r^{w_r})^{(i+1)})\geq\reg\,(S/(I_r^{w_r})^{(i)})+1$ for any $i\in[t-1]$ by Lemma \ref{colon}.

If $w=\max\{w(x) \,|\, x\in V_2\}$, then the first inequality in the expression of $\reg\left(\frac{S}{(J+I_r^{w_r})^{(t)}}\right)$ becomes equality by the induction, and the last inequality becomes equality because the
function $g(j):=\sum\limits_{x\in V}{w(x)}-|V|+1+(t-j)w+(j-1)w'+(t-j)$ is strictly monotonic decreasing with respect to $j$.

Hence, the desired results hold by using Lemma \ref{exact} (1) to the following exact sequence
\[
0\longrightarrow S/I(D)^{(t)}\longrightarrow S/(J+I_r^{w_r})^{(t)}\oplus S/\widehat I_r^{\,(t)}\longrightarrow S/((J+I_r^{w_r})^{(t)}+\widehat I_r^{\,(t)})\longrightarrow 0.\qedhere
\]
\end{proof}

An immediate consequence of the above theorem and  Theorem \ref{regularity} is the following corollary.

\begin{Corollary}
	\label{cor1}
	Let  $D:=D_1*\cdots*D_r$ be the join  of   weighted oriented graphs  $D_1,\ldots,D_r$ as in  Theorem \ref{symbolic isolate2}, then, for any $t\ge 1$, we have
	\[
	\reg\,(S/I(D)^{(t)})\leq \reg\,(S/I(D))+(t-1)(w+1).
	\]
	And these	equalities hold  if $w=\max\{w(x) \,|\, x\in V_2\}$, where  $w=\max\{w(x) \,|\, x\in V(D)\}$.
\end{Corollary}

\medskip
The following examples show that the upper bounds in  Theorems 	\ref{symbolic isolate2} can be obtained, but may be strict.
\begin{Example}
	\label{exam2}
	Let $I(D)=(x_1x_3^2,x_1x_4^2,x_1x_5^3,x_2x_3^2,x_2x_4^2,x_2x_5^3,x_3x_4^2,x_3x_5^3)$ be the edge ideal of the join of  weighted oriented  graphs with the partition $V_1=\{x_1,x_2\}, V_2=\{x_3\}, V_3=\{x_4,x_5\}$. The  weight function is $w_1=w_2=1, w_3=w_4=2$ and $w_5=3$. By using CoCoA, we obtain $\reg\,(S/I(D)^{(2)})=8$. But we have $\sum\limits_{i=1}^5{w_i}-5+1+(3+1)=9$.
\end{Example}

\begin{Example}
	\label{exam3}
	Let $I(D)=(x_1x_3^5,x_1x_4^2,x_1x_5^3,x_2x_3^5,x_2x_4^2,x_2x_5^3,x_3x_4^2,x_3x_5^3)$ be the edge ideal of the join of  weighted oriented  graphs with the partition $V_1=\{x_1,x_2\}, V_2=\{x_3\}, V_3=\{x_4,x_5\}$.  The weight function is $w_1=w_2=1, w_3=5, w_4=2$ and $w_5=3$. Thus  $w=\max\{w_i \,|\, i\in [5]\}=5$. By using CoCoA, we  have $\reg\,(S/I(D)^{(2)})=\sum\limits_{i=1}^5{w_i}-5+1+(5+1)=14$.
\end{Example}

\medskip
\section{Ordinary powers of the edge ideal of the join of graphs with at least an  oriented edge}
In this section, we study  depth and regularity of ordinary powers of the edge ideal of the join of weighted oriented graphs with at least an  oriented edge  for per graph.
We give the exact formulas for depth and regularity of the edge ideal of the join of two weighted oriented graphs, and  also provide the upper  bounds of  regularity of  ordinary powers when the second  graph is a
weighted oriented complete graph.

\begin{Lemma}\label{bipartitepower2}
	Let $D:=D_1*D_2$ be the join of two weighted oriented graphs  $D_1$ and $D_2$, where the vertex set of $D_i$ is  $V_i=\{x_{ij}\,|\, j\in [n_i]\}$ and $n_i=|V_i|$ for  $i\in [2]$.
If each $D_i$ contains at least an  oriented edge.
 Then
	\begin{itemize}
		\item[(1)] $\depth\,(S/I(D))=1$;
		\item[(2)]  $\reg\,(S/I(D))=\max\{\sum\limits_{x\in V_2}{w(x)}-|V_2|+\reg\,(S_1/I(D_1)),\reg\,(S_2/I(D_2))\}$,
	 where $S_1=k[x_{11}, \ldots,x_{1n_1}]$ and $S_2=k[x_{21},\ldots,x_{2n_2}]$.
\end{itemize}	
	\end{Lemma}
\begin{proof} We use  notations of Lemma \ref{decompose}. Obviously, we have  $I(D)=(\widehat I_1+I(D_1))\cap (\widehat I_2+I(D_2))$
	by Theorem \ref{rdecompose}.
Set $J_i=\widehat I_i+I(D_i)$ for $i\in [2]$, then  $I(D)=J_1\cap J_2$ and   $\supp\,(\widehat{I}_i)\cap \supp\,(I(D_i))=\emptyset$. Thus,
by Lemma \ref{sum1} (2), one has
\begin{eqnarray*}
	\depth\,(S/J_1)&=&\depth\,(S_1/I(D_1))+\depth\,(S_2/\widehat I_1)=\depth\,(S_1/I(D_1)),\\
\depth\,(S/J_2)&=&\depth\,(S_2/I(D_2))+\depth\,(S_1/\widehat I_2)=\depth\,(S_2/I(D_2)).
\end{eqnarray*}
By  Lemmas \ref{quotient} (1) and \ref{sum1} (1), we have
			\begin{eqnarray*}
				\reg\,(S/J_1)&=&\reg\,(S_2/\widehat I_1)+\reg\,(S_1/I(D_1))=\sum\limits_{x\in V_2}{w(x)}-| V_2|+\reg\,(S_1/I(D_1)),\\
		\reg\,(S/J_2)&=&\reg\,(S_1/\widehat I_2)+\reg\,(S_2/I(D_2))=\reg\,(S_2/I(D_2)).
	\end{eqnarray*}
Note that $(J_1,J_2)=\widehat I_1+\widehat I_2=(x_{11}, \ldots,x_{1n_1},x_{21}^{w_{21}},\ldots,x_{2n_2}^{w_{2n_2}})$. Hence one has
\[
\depth\,(S/(J_1,J_2))=0\text{ and\ } \reg\,(S/(J_1,J_2))=\sum\limits_{x\in V_2}{w(x)}-|V_2|.
\]
Applying lemma \ref{exact} (1), (4) to the following exact sequence
\[
0\longrightarrow S/J_1\cap J_2\longrightarrow S/J_1\oplus S/J_2\longrightarrow S/(J_1,J_2)\longrightarrow 0,
\]
we obtain the desired results.
\end{proof}

\begin{Theorem}\label{r-partitedirected}
	Let $r\geq3$ be an integer, and let $D:=D_1*\cdots*D_r$ be the join of  weighted oriented graphs  $D_1,\ldots,D_r$,  where the vertex set of $D_i$ is  $V_i=\{x_{ij}\,|\, j\in [n_i]\}$ and $n_i=|V_i|$ for  $i\in [r]$. If each $D_i$ contains at least an  oriented edge. Then
	\begin{itemize}
		\item[(1)] $\depth\,(S/I(D))=1$;
		\item[(2)]  $\reg\,(S/I(D))\!=\!\max\{\reg\,(S_i/I(D_i))+\sum\limits_{x\in T_i}w(x)-|T_i|:i\in [r]\}$, where   $T_i=V(D)\setminus \bigcup\limits_{j=1}^i V_j$ for any  $i\in [r]$.
\end{itemize}	
\end{Theorem}	
\begin{proof} We apply induction on  $r$. The case $r=2$ follows from Lemma \ref{bipartitepower2}.
Now, assume that $r\geq  3$. In this case, we  have $D=(D_1*\cdots*D_{r-1})*D_r$ and $I(D)=I((D_1*\cdots*D_{r-1})*D_r)$.
Thus, by Lemma \ref{bipartitepower2} and the inductive hypothesis, we obtain that
	$\depth\,(S/I(D))=1$ and
\begin{align*}
& \reg\left(\frac{S}{I(D)}\right)=\reg\left(\frac{S}{I((D_1*\cdots*D_{r-1})*D_r)}\right)\\
&=\max\left\{\sum\limits_{x\in V_r}{w(x)}-|V_r|+\reg\left(\frac{S'}{I(D_1*D_2*\cdots*D_{r-1})}\right),\reg\left(\frac{S_r}{I(D_r)}\right)\right\}\\
&=\max\left\{{\sum\limits_{x\in V_r}{w(x)}-|V_r|+\max\left\{\sum\limits_{x\in T'_i}{w(x)}-|T'_i|+\reg\left(\frac{S_i}{I(D_i)}\right):i\in [r-1]\right\},}\right.\\
& \left.{\reg\left(\frac{S_r}{I(D_r)}\right)}\right\}\\
&=\max\left\{\reg\left(\frac{S_i}{I(D_i)}\right)+\sum\limits_{x\in T_i}w(x)-|T_i|:i\in [r]\right\}	
\end{align*}
 where $S'=S_1\otimes_\KK S_2\otimes_\KK\cdots \otimes_\KK S_{r-1}$ and  $T'_i=V(D)\setminus (V_r\cup(\bigcup\limits_{j=1}^i V_j))$.
	\end{proof}		

\begin{Remark}\label{general}
	Let $r\geq 2$ be an integer, and let $D:=D_1*\cdots*D_r$ be the join of  weighted oriented graphs  $D_1,\ldots,D_r$ with  $V_i=\{x_{ij}\,|\, j\in [n_i]\}$ and $n_i=|V_i|$ for  $i\in [r]$.  If there exist some  $D_i$ consisting of  isolated vertices and some $D_j$  containing at least an  oriented edge. From the proof of Lemma \ref{bipartitepower2} and Theorem \ref{r-partitedirected}, we  still have $\depth\,(S/I(D))=1$. But we can't guarantee  the regularity of $S/I(D)$ to obtain the  equality, that is, we have  $\reg\,(S/I(D))\le\max\{\sum\limits_{x\in T_1}w(x)-|T_1|+1,\reg\,(S_i/I(D_i))+ \sum\limits_{x\in T_i}w(x)-|T_i|: i\in [r] \}$.
\end{Remark}

An immediate consequence of Theorem \ref{r-partitedirected} is  the generalizations of \cite[Corollary 3.10, Proposition 3.12]{M}.
\begin{Corollary}\label{rpartitepower2}
Let  $G:=G_1*\cdots*G_r$ be the join  of    graphs  $G_1,\ldots,G_r$. If each $G_i$  contains at least an  edge. Then
\begin{itemize}
		\item[(1)] $\depth\,(S/I(G))=1$;
		\item[(2)]  $\reg\,(S/I(G))=\max\{\reg\,(S_i/I(G_i))\mid i\in [r]\}$.
\end{itemize}	
\end{Corollary}
\begin{proof}
This is a direct consequence of the above theorem in the case of  all vertices of $D_i$ having trivial weights.
	\end{proof}		

\begin{Theorem}
	\label{complete graphs}
	Let  $K_n$ be a weighted oriented complete graph with edge ideal $I(K_n)=(x_ix_j^{w_j}: i,j\in [n] \text{\ and\ } i<j)$. Then, for any $t\ge 1$, we have
	\[
	\reg\,(S/I(K_n)^t)\leq \sum\limits_{i=1}^n{w_i}-n+1+(t-1)(w+1)
	\] where  $w=\max\{w_i \,|\, i\in [n]\}$. These	equalities hold here when $w=w_2$.
\end{Theorem}
\begin{proof}
	We apply induction on  $n$ and $t$. The case $n=2$ is trivial and the  case $t=1$ follows from Theorem \ref{regularity} (2). Now, assume that $n\geq 3$ and $t\geq 2$.
	Let $J=(x_1,\ldots, x_{n-1})$, then $I(K_n)^t=(I(K_n\setminus x_{n})+Jx_n^{w_n})^{t}=\sum\limits_{i=1}^tI(K_n\setminus x_{n})^{t-i}(Jx_n^{w_n})^i+I(K_n\setminus x_{n})^{t}$.
	Let $P_0=I(K_n)^t$, $P_j=I(K_n)^t:(x_n^{w_n})^j$ and $Q_{j-1}=P_{j-1}+(x_n^{w_n})$ for any $j\in [t]$.
	Since $I(K_n\setminus x_n)\subseteq J$, we have
	$P_j=\sum\limits_{i=j+1}^t(J^ix_{n}^{w_{n}(i-j)}I(K_n\setminus x_{n})^{t-i})+J^jI(K_n\setminus x_{n})^{t-j}$
	with the convention $P_t=J^t$, and $Q_{j-1}=J^{j-1}I(K_n\setminus x_{n})^{t-j+1}+(x_{n}^{w_{n}}).$
	Thus, by Lemma \ref{sum1} (1) and the inductive hypothesis, we have
	\begin{align*}
		\reg\,(S/Q_0)&=\reg\,(Q_0)-1=\reg\,(I(K_n\setminus x_{n})^t+(x_n^{w_n}))-1\\
		&=\reg\,(I(K_n\setminus x_{n})^t)+w_n-2\\
		&\leq\sum\limits_{i=1}^{n-1}{w_i}-(n-1)+2+(t-1)(w'+1)+w_n-2\\
		&=\sum\limits_{i=1}^{n}{w_i}-n+1+(t-1)(w'+1)\\
		&\le \sum\limits_{i=1}^{n}{w_i}-n+1+(t-1)(w+1),
	\end{align*}
	where $w'=\max\{w_i \,|\, i\in[n-1]\}$.
	By Lemma \ref{regular}, we have \begin{align*}
		\reg\,((S/P_t)(-tw_n)) &=\reg\,(S/J^t)+tw_n=(t-1)(w_n+1)+w_n\\
		&\leq\sum\limits_{i=1}^{n}{w_i}-n+1+(t-1)(w+1).
	\end{align*}
	Note that $\dim\,(S/J^j)=1$	for any $j\in[t-1]$. Then, by Lemmas \ref{product}, \ref{quotient} (1),   \ref{sum1} (1) and the inductive hypothesis, we have
	\begin{align*}
		\reg\,((S/Q_j)(-jw_n)) &=\reg\,(J^jI(K_n\setminus x_{n})^{t-j}+(x_n^{w_n}))-1+jw_n\\
		&\leq\reg\,(J^j)+\reg\,(I(K_n\setminus x_{n})^{t-j})+w_n-2+jw_n\\
		&\leq\sum\limits_{i=1}^{n}{w_i}-n+1+(t-j-1)(w'+1)+j(w_n+1)\\
		&\leq\sum\limits_{i=1}^{n}{w_i}-n+1+(t-1)(w+1).
	\end{align*}
	where the last inequality holds since $w_n,w'\leq w$. Using  Lemma \ref{exact} (1) to the following exact sequences  by shifting
	\begin{gather*}
		\hspace{1.0cm}\begin{matrix}
			0 & \longrightarrow & \frac{S}{P_1}(-w_n)  & \stackrel{ \cdot x_{n}^{w_n}} \longrightarrow  &  \frac{S}{I(K_n)^t} & \longrightarrow &\frac{S}{Q_0} & \longrightarrow & 0 & \\
			0 & \longrightarrow & \frac{S}{P_2}(-w_n) & \stackrel{ \cdot x_{n}^{w_n}} \longrightarrow & \frac{S}{P_1} &\longrightarrow & \frac{S}{Q_1} & \longrightarrow & 0  \\
			&  &\vdots&  &\vdots&  &\vdots&  &\\
			0 & \longrightarrow & \frac{S}{P_{t-1}}(-w_n) & \stackrel{ \cdot x_{n}^{w_n}} \longrightarrow & \frac{S}{P_{t-2}} &\longrightarrow & \frac{S}{Q_{t-2}} & \longrightarrow & 0  \\
			0 & \longrightarrow & \frac{S}{P_t}(-w_n) & \stackrel{ \cdot x_{n}^{w_n}} \longrightarrow & \frac{S}{P_{t-1}} &\longrightarrow & \frac{S}{Q_{t-1}} & \longrightarrow & 0,
		\end{matrix}
	\end{gather*}
	we get the desired results,
	
	If   $w_2=w$, then, by the same technique, we can get $\reg\,(S/Q_0)=\sum\limits_{i=1}^{n}{w_i}-n+1+(t-1)(w+1)$. Hence, the expected equality also holds in this case.
\end{proof}	

The following examples show that the upper bounds in  Theorems 	\ref{complete graphs} can be obtained, but may be strict.
\begin{Example}
	\label{exam4}
	Let $I(K_4)=(x_1x_2^7,x_1x_3^3,x_1x_4^6,x_2x_3^3,x_2x_4^6,x_3x_4^6)$ be the edge ideal of a weighted oriented complete graph $K_4$, its weight function is $w_1=1, w_2=7, w_3=3$ and $w_4=6$.  By using CoCoA, we obtain $\reg\,(S/I(K_4)^2)=\sum\limits_{i=1}^4{w_i}-4+1+(7+1)=22$, where   $w=7$.		
\end{Example}

\begin{Example}
	\label{exam5}
	Let $I(K_4)=(x_1x_2^2,x_1x_3^3,x_1x_4^6,x_2x_3^3,x_2x_4^6,x_3x_4^6)$ be the edge ideal of a weighted oriented complete graph $K_4$, its weight function is $w_1=1, w_2=2, w_3=3$ and $w_4=6$.  By using CoCoA, we obtain $\reg\,(S/I(K_4)^2)=15$. However, $\sum\limits_{i=1}^4{w_i}-4+1+(2-1)(w+1)=16$, where   $w=6$.
	\end{Example}

\begin{Lemma}\label{isolate point}
	Let $D=D_1*D_2$ be the join of two weighted oriented graphs  $D_1$ and $D_2$, where $V(D_i)=\{x_{ij}\,|\, j\in [n_i]\}$  and $n_i=|V(D_i)|$ for  $i\in [2]$. If $n_2=1$,
	then  \[
	\reg\,(S/I(D))\geq\reg\,(S_1/I(D_1))+w_{21}-1.
	\]
	where $S_1=\KK[x_{11} \ldots,x_{1n_1}]$. The equality holds when $D_1$ contains at least an  oriented edge.
\end{Lemma}

\begin{proof}
	The case  $D_1$ contains of isolated vertices is obvious by Theorem \ref{bipartitepower} (2).
	We can  assume that $D_1$ contains at least an  oriented edge. Since $I(D):x_{21}^{w_{21}}=(x_{1i}\,|\, i\in [n_1])$ and $I(D)+(x_{21}^{w_{21}})=I(D_1)+(x_{21}^{w_{21}})$, we have
	$\reg\,((S/(I(D):x_{21}^{w_{21}}))(-w_{21})) =w_{21}$ and $\reg\,(S/(I(D)+(x_{21}^{w_{21}}))) =\reg\,(S_1/I(D_1))+w_{21}-1$.
	We can obtain the desired result  by a fact that $\reg\,(S_1/I(D_1))\ge 1$ and by using Lemma \ref{exact} (1) to	the following exact sequence
	\[
	0\longrightarrow  (S/I(D):x_{21}^{w_{21}})(-w_{21}) \stackrel{\cdot x_{21}^{w_{21}}}\longrightarrow S/I(D)\longrightarrow S/(I(D),x_{21}^{w_{21}})\longrightarrow 0.\qedhere
	\]
\end{proof}

Now we are ready to present the main result of this section.
\begin{Theorem}
	\label{power3}
	Let  $D=D_1*D_2$ be the join of two weighted oriented graphs  $D_1$ and $D_2$,  where  $V(D_i)=\{x_{ij}\,|\, j\in [n_i]\}$ and $n_i=|V(D_i)|$ for  $i\in [2]$. If $D_2$ is a weighted oriented complete graph as Theorem \ref{complete graphs} and $\reg\,(S_1/I(D_1)^t)\leq \reg\,(S_1/I(D_1))+(t-1)(w'+1)$, where $S_1=\KK[x_{1i}: i\in V(D_1)]$ and $w'=\max\{w(x) \,|\, x\in V(D_1)\}$.
	Then, for any integer $t\ge 1$, we have
	\[
	\reg\,(S/I(D)^t)\leq \reg\,(S/I(D))+(t-1)(w+1)
	\]
where $w=\max\{w(x) \,|\, x\in V(D_1)\cup V(D_2)\}$. The equality holds  when $w'=w$ and $\reg\,(S_1/(I(D_1)^t))= \reg\,(S_1/I(D_1))+(t-1)(w'+1)$.
\end{Theorem}
\begin{proof}
	We apply induction on  $t$ and $n_2$. The case $t=1$ is trivial and the $n_2=1$ case  follows from Lemma \ref{isolate}.
	Now,  suppose $t\geq2$ and $n_2\geq2$.  Let $P_0=I(D)^t$, $P_j=I(D)^t:(x_{2n_2}^{w_{2n_2}})^j$ and $Q_{j-1}=P_{j-1}+(x_{2n_2}^{w_{2n_2}})$ for any $j\in [t]$, and  set $J=(x_{1i}\,|\,i\in[n_1])+(x_{2i}\,|\,i\in[n_2-1])$. By some  simple calculations, we can get
	$P_j=\sum\limits_{i=j+1}^tJ^ix_{2n_2}^{(i-j)w_{2n_2}}I(D\setminus x_{2n_2})^{t-i}+J^jI(D\setminus x_{2n_2})^{t-j}$
	with the convention $P_t=J^t$, and
	$Q_{j-1}=J^{j-1}I(D\setminus x_{2n_2})^{t-j+1}+(x_{2n_2}^{w_{2n_2}})$.
	
	By Lemmas \ref{sum1} (1), \ref{isolate point}  and the inductive hypothesis, we have
	\begin{align*}
		\reg\,(S/Q_0)&=\reg\,(S/I(D\setminus x_{2n_2})^t)+w_{2n_2}-1\\
		&\leq \reg\,(S/I(D\setminus x_{2n_2}))+(t-1)(w''+1)+w_{2n_2}-1\\
		&= \reg\,(S/I(D))+(t-1)(w''+1)\\
		&\leq \reg\,(S/I(D))+(t-1)(w+1),
	\end{align*}
	where $w''=\max\{w(x) \,|\, x\in V(D\setminus x_{2n_2})\}$.
	By Lemma \ref{regular}, we obtain that
	\begin{align*}
		\reg\,((S/P_t)(-tw_{2n_2})) &=\reg\,(S/J^t)+tw_{2n_2}=(t-1)(w_{2n_2}+1)+w_{2n_2}\\
		&\leq \reg\,(S/I(D))+(t-1)(w+1).
	\end{align*}
	where the   inequality holds since $w\geq w_{2n_2}$ and $\reg\,(S/I(D))\geq w_{2n_2}$.
	
	Note that $\dim\,(S/J^j)=1$	for any $j\in[t-1]$. Then, by Lemmas \ref{product}, \ref{quotient} (1),   \ref{sum1} (1) and the inductive hypothesis, we have
	\begin{align*}
		\reg\,((S/Q_j)(-jw_{2n_2}))
		&=\reg\,(J^jI(D\setminus x_{2n_2})^{t-j}+(x_{2n_2}^{w_{2n_2}}))-1+jw_{2n_2}\\
		&=\reg\,(J^jI(D\setminus x_{2n_2})^{t-j})+w_{2n_2}-2+jw_{2n_2}\\
		&\leq\reg\,(J^j)+\reg\,(I(D\setminus x_{2n_2})^{t-j})+w_{2n_2}-2+jw_{2n_2}\\
		&=j+\reg\,(S/I(D\setminus x_{2n_2})^{t-j})+w_{2n_2}-1+jw_{2n_2}\\
		&\leq\reg\,(S_1/I(D\setminus x_{2n_2}))+(t-j-1)(w''+1)\\
		&+w_{2n_2}-1+j(w_{2n_2}+1)\\
		&=\reg\,(S/I(D))+(t-j-1)(w''+1)+j(w_{2n_2}+1)\\
		&\leq\reg\,(S/I(D))+(t-1)(w+1)
	\end{align*}
	where the last inequality holds since $w_{2n_2}, w''\leq w$, and  the penultimate equality holds because  $D\setminus x_{2n_2}$ contains at least an  oriented edge.

	In particular, if  $w'=w$ and $\reg\,(S_1/(I(D_1)^t))= \reg\,(S_1/I(D_1))+(t-1)(w'+1)$, it follows from  Lemmas \ref{sum1} (1), \ref{isolate point} and  the inductive hypothesis that
		$\reg\,(S/Q_0)=\reg\,(S/I(D))+(t-1)(w+1)$.

	We can confirm the
	assertion by applying  Lemma \ref{exact} (1) to the following exact sequences
	\begin{gather}
		\hspace{1.0cm}\begin{matrix}
			0 & \longrightarrow & \frac{S}{P_1}(-w_{2n_2})  & \stackrel{ \cdot x_{2n_2}^{w_{2n_2}}} \longrightarrow  &  \frac{S}{I(D)^t} & \longrightarrow &\frac{S}{Q_0} & \longrightarrow & 0 & \\
			0 & \longrightarrow & \frac{S}{P_2}(-w_{2n_2}) & \stackrel{ \cdot x_{{2n_2}}^{w_{2n_2}}} \longrightarrow & \frac{S}{P_1} &\longrightarrow & \frac{S}{Q_1} & \longrightarrow & 0  \\
			&  &\vdots&  &\vdots&  &\vdots&  &  \label{eqn:123-1}\\
			0 & \longrightarrow & \frac{S}{P_{t-1}}(-w_{2n_2}) & \stackrel{ \cdot x_{2n_2}^{w_{2n_2}}} \longrightarrow & \frac{S}{P_{t-2}} &\longrightarrow & \frac{S}{Q_{t-2}} & \longrightarrow & 0  \\
			0 & \longrightarrow & \frac{S}{P_t}(-w_{2n_2}) & \stackrel{ \cdot x_{2n_2}^{w_{2n_2}}} \longrightarrow & \frac{S}{P_{t-1}} &\longrightarrow & \frac{S}{Q_{t-1}} & \longrightarrow & 0.
		\end{matrix}
	\end{gather}
\end{proof}	

\begin{Lemma}
\label{isolate}
With the assumptions and notation of Theorem 
\ref{power3}. If $n_2=1$, then
\[
\reg\,(S/I(D)^t)\leq \reg\,(S/I(D))+(t-1)(w+1)
\]
where $w=\max\{w(x) \,|\, x\in V_1\cup V_2\}$.
These equalities hold  when $w'=w$ and $\reg\,(S_1/(I(D_1)^t))= \reg\,(S_1/I(D_1))+(t-1)(w'+1)$ for all $t\geq1$.
\end{Lemma}
\begin{proof} The proof will be essentially the same as that for Theorem \ref{power3}.
	We apply induction on  $t$ with the $t=1$  case being trivial.   Now,  suppose $t\geq2$.  Let   $J=(x_{1i}\,|\,i\in[n_1])$, $P_0=I(D)^t$, $P_j=I(D)^t:(x_{21}^{w_{21}})^j$ and $Q_{j-1}=P_{j-1}+(x_{21}^{w_{21}})$ for any $j\in [t]$. \\
	By Lemmas \ref{sum1} (1) and \ref{isolate point} , we have
	\begin{align*}
		\reg\,(S/Q_0)&=\reg\,(S_1/(I(D\setminus x_{21})^t)+w_{21}-1\\
		&=\reg\,(S_1/(I(D_1)^t)+w_{21}-1\\
		&\leq\reg\,(S_1/I(D_1))+(t-1)(w'+1)+w_{21}-1\\
		&\leq\reg\,(S/I(D))+(t-1)(w'+1)\\
		&\leq\reg\,(S/I(D))+(t-1)(w+1)
	\end{align*}
	where the first inequality holds by the assumption that $\reg\,(S_1/(I(D_1)^t))+w_{21}-1
	\leq\reg\,(S_1/I(D_1))+(t-1)(w'+1)+w_{21}-1$.\\
	By some similar arguments as the proof of Theorem \ref{power3}, we can get	
	\begin{align*}
		\reg\,((S/P_t)(-tw_{2n_2}))&\leq \reg\,(S/I(D))+(t-1)(w+1),\\
		\reg\,((S/Q_j)(-jw_{21}))&\leq\reg\,(S/I(D))+(t-1)(w+1) \text{\ for\ } j=1,\ldots,t-1.
	\end{align*}
	In particular, if  $w'=w$ and $\reg\,(S_1/(I(D_1)^t))= \reg\,(S_1/I(D_1))+(t-1)(w'+1)$ for all $t\geq1$, then
	$D_1$ contains at least an  oriented edge, because on the contrary, we have  $\reg\,(S_1/(I(D_1)^t))=0$. It follows from  Lemmas \ref{sum1} (1)  and \ref{isolate point} that
	$\reg\,(S/Q_0)=\reg\,(S/I(D))+(t-1)(w+1)$.
	Thus we can obtain the assertion by applying Lemma \ref{exact} (1) to the  exact sequences (\ref{eqn:123-1}).
\end{proof}

\begin{Remark}
\label{rem1}The conditions that $\reg\,(S_1/I(D_1)^t)\leq \reg\,(S_1/I(D_1))+(t-1)(w'+1)$ in Theorems 	\ref{power3}  can be obtained for many weighted oriented graphs, such as  weighted rooted forests,  naturally weighted  oriented cycles,   the disjoint union of some weighted oriented  gapfree bipartite graph and so on
 {\em (see \cite[Theorem 4.1]{XZWZ}, \cite[Theorem 4.5]{WZX},  \cite[Theorem 4.9]{ZXWZ}).}
\end{Remark}

As a consequence of Lemma \ref{isolate}, we have the following:
\begin{Corollary}\label{wheel graph}
Let $D_1$ be a naturally weighted  oriented cycle with $w(x)\geq2$ for all $x\in V(D_1)$ and let $D_2$ be a graph composed of a isolated vertex. If $D=D_1*D_2$, i.e.,
$D$ is a weighted oriented wheel graph, and   $w:=\max\{w(x) \,|\, x\in V(D)\}=\max\{w(x) \,|\, x\in V(D_1)\}$.
Then
\[
\reg\,(S/I(D)^t)=\sum\limits_{x\in V(D)}{w(x)}-|V(D)|+1+(t-1)(w+1).
\]
\end{Corollary}
\begin{proof}   By \cite[Theorem 4.5]{WZX}, we have $\reg\,(S_1/(I(D_1)^t))= \reg\,(S_1/I(D_1))+(t-1)(w+1)$ for all $t\geq1$.
 The expected formulas follow from 	Theorem \ref{power3}.
\end{proof}		

The following examples show that the upper bounds in  Theorems 	\ref{power3} can be obtained, but may be strict.

\begin{Example}
		\label{exam6}
		Let $I(D)=(x_1x_2^2,x_2x_3^2,x_3x_1^4,x_1y^3,x_2y^3,x_3y^3)$ be the edge ideal of a weighted oriented wheel graph $D$, its weight function is $w(x_1)=4, w(x_2)=w(x_3)=2$ and $w(y)=3$. By using CoCoA, we have $\reg\,(S/I(D))=7$, $\reg\,(S/I(D)^2)=12$. Thus $\reg\,(S/I(D)^2)=\reg\,(S/I(D))+(w+1)$, where $w=4$.		
\end{Example}

\begin{Example}
  \label{exam7}
  Let $I(D)=(x_1x_2^2,x_2x_3^2,x_3x_1^2,x_1y^3,x_2y^3,x_3y^3)$ be the edge ideal of a weighted oriented wheel graph $D$, its weight function is $w(x_1)=w(x_2)=w(x_3)=2$ and $w(y)=3$. By using CoCoA, we have $\reg\,(S/I(D))=5$, $\reg\,(S/I(D)^2)=8$. Thus $\reg\,(S/I(D)^2)<\reg\,(S/I(D))+(w+1)$.
  \end{Example}

\medskip
\hspace{-6mm} {\bf Acknowledgments}

 \vspace{3mm}
\hspace{-6mm} The authors are grateful to the computer algebra system CoCoA \cite{Co} for providing us with a large number of examples.  This research is supported by the Natural
Science Foundation of Jiangsu Province (No. BK20221353)  and   foundation of the Priority Academic Program Development of Jiangsu Higher Education Institutions. The third author is supported by the National Natural Science Foundation of China (No. 12126330).

\medskip

\end{document}